\documentclass[a4paper,11pt]{article}
\usepackage{graphicx} % standard LaTeX graphics tool
\usepackage{amsfonts,amsmath,amsthm,color}
\newtheorem{theorem}{Theorem}[section] % 1st argument is your name for it
\newtheorem{lemma}[theorem]{Lemma}     % 2nd argument is what is printed
\newtheorem{corollary}[theorem]{Corollary}

\newtheorem{definition}[theorem]{Definition}
\newtheorem{example}[theorem]{Example}

\begin{document}

\title{Semispectral Measures and Feller Markov Kernels}% end with percent
\author{Roberto Beneduci\thanks{e-mail rbeneduci@unical.it}\\
{\em Dipartimento di Matematica,}\\
{\em Universit\`a della Calabria,}\\
and\\
{\em Istituto Nazionale di Fisica Nucleare, Gruppo c. Cosenza,}\\
}
\date{}
\maketitle

%Insert `2000 Mathematics Subject Classification' numbers her
%\classification{11B83 (primary), 11J71, 37A45, 60G10 (secondary).}
% Please refer to {\tt http://www.ams.org/msc/} for a list of codes}

%\extraline{Acknowledgements of grants and financial support should
%be included here; more general \textsl{Acknowledgements} are better
%placed either immediately before the bibliography (see
%page~\pageref{ackref}) or at the end of the introduction. Since
%author names should not carry footnote marks, instead refer to `The
%first author', etc. No `keywords' should be supplied. This guide was
%last revised on \filedate\ and documents {\tt lms.cls} version
%\fileversion.}

\maketitle

\begin{abstract}
We give a characterization of commutative semispectral measures by means of Feller and Strong Feller Markov kernels. In particular:

\begin{itemize}
\item we show that a semispectral measure $F$ is commutative if and only if  there exist a self-adjoint operator $A$ and a Markov kernel $\mu_{(\cdot)}(\cdot):\Gamma\times\mathcal{B}(\mathbb{R})\to[0,1]$, $\Gamma\subset\sigma(A)$, $E(\Gamma)=\mathbf{1}$, such that 
$$F(\Delta)=\int_{\Gamma}\mu_{\Delta}(\lambda)\,dE_{\lambda},$$
\noindent
 and $\mu_{(\Delta)}$ is continuous for each $\Delta\in R$ where, $R\subset\mathcal{B}(\mathbb{R})$ is a ring which generates the Borel $\sigma$-algebra of the reals $\mathcal{B}(\mathbb{R})$. Moreover, $\mu_{(\cdot)}(\cdot)$ is a Feller Markov kernel and separates the points of $\Gamma$.

\item we prove that $F$ admits a strong Feller Markov kernel $\mu_{(\cdot)}(\cdot)$, if and only if $F$ is uniformly continuous. Finally, we prove that if $F$ is absolutely continuous with respect to a regular finite measure $\nu$ then, it admits a strong Feller Markov kernel.
\end{itemize}
The mathematical and physical relevance of the results is discussed giving a particular emphasis to the connections between $\mu$ and the imprecision of the measurement apparatus.
\end{abstract}

%\part{Use this type of header for very long papers only}
% use lowercase except for proper names

\section{Introduction}
A real semispectral measure (or Positive operator Valued measure) is a map $F:\mathcal{B}(\mathbb{R})\to\mathcal{L}_s^+(\mathcal{H})$ from the Borel $\sigma$-algebra of the reals to the space of positive self-adjoint operators on a Hilbert space $\mathcal{H}$. If, $F(\Delta)$ is a projection operator for each $\Delta\in\mathcal{B}(\mathbb{R})$, $F$ is called spectral measure (or Projection Valued measure). Therefore, the set of spectral measures is a subset of the set of semispectral measures. Moreover, spectral measures are in one-to-one correspondence with self-adjoint operators (spectral theorem) \cite{Reed} and are used in standard quantum mechanics to represent quantum observables. It was pointed out \cite{Ali,Bush,Davies,Holevo1,Prugovecki,Schroeck} that semispectral measures are more suitable than spectral measures in representing quantum observables. The quantum observables described by semispectral measures are called generalized observables or unsharp observables and play a key role in quantum information theory, quantum optics, quantum estimation theory \cite{Bush,Helstrom,Holevo1,Schroeck1} and in the phase space formulation of quantum mechanics \cite{Schroeck1,B9,B10}. It is then natural to ask what are the relationships between semispectral and spectral measures. A clear answer can be given in the commutative case \cite{Ali,B0,B1,B2,B3,B4,B5,B6,B7,B8,Holevo,P}. Indeed \cite{B1,P}, a real positive semispectral measure $F$ is commutative if and only if there exist a bounded self-adjoint operator $A$ and a Markov kernel (transition probability) $\mu_{(\cdot)}(\cdot):\sigma(A)\times\mathcal{B}(\mathbb{R})\to[0,1]$ such that 
$$F(\Delta)=\int_{\sigma(A)}\mu_{\Delta}(\lambda)\,dE_{\lambda}$$
\noindent
 where, $E$ is the spectral measure corresponding to $A$. %The necessary and sufficient condition holds also if we use weak Markov kernels instead of Markov kernels \cite{P}.
In other words, $F$ is a smearing of the spectral measure $E$ corresponding to $A$. 

\noindent
As an example we can consider the following unsharp position observable 
\begin{align}\label{PositionExample}
\langle\psi,Q^f(\Delta)\psi\rangle&:=\int_{[0,1]}\mu_{\Delta}(x)\,d\langle\psi,Q_x\psi\rangle,\quad\Delta\in\mathcal{B}(\mathbb{R}),\quad\psi\in L^2([0,1]),\\
\mu_{\Delta}(x)&:=\int_{\mathbb{R}}\chi_{\Delta}(x-y)\, f(y)\,dy,\quad x\in[0,1]\notag
\end{align}
\noindent
where, $f$ is a positive, bounded, Borel function such that $f(y)=0$, $y\notin [0,1]$, $\int_{[0,1]} f(y)dy=1$, and $Q_x$ is the spectral measure corresponding to the position operator
\begin{align*}
Q:L^2([0,1])&\to L^2([0,1])\\
\psi(x)&\mapsto Q\psi:=x\psi(x)
\end{align*}
\noindent
We recall that $\langle\psi,Q(\Delta)\psi\rangle$ is interpreted as the probability that a perfectly accurate measurement (sharp measurement) of the position gives a result in $\Delta$.
\noindent
Then, a possible interpretation of equation (\ref{PositionExample}) is that $Q^f$ is a randomization of $Q$. Indeed \cite{Prugovecki}, the outcomes of the measurement of the position of a particle depend on the measurement imprecision\footnote{There are other possible interpretations of the randomization. For example, it could be due to the existence of a no-detection probability depending on hidden variables \cite{Garola}.} so that, if the sharp value of the outcome of the measurement of $Q$ is $x$ then the apparatus produces with probability $\mu_{\Delta}(\lambda)$ a reading in $\Delta$. 

It is worth noticing that (see example \ref{example2} in section 5) the Markov kernel 
$$\mu_{\Delta}(x):=\int_{\mathbb{R}}\chi_{\Delta}(x-y)\, f(y)\,dy,\quad x\in[0,1]$$
\noindent
 in equation (\ref{PositionExample}) above is such that the function $x\mapsto\mu_{\Delta}(x)$ is continuous for each $\Delta\in\mathcal{B}(\mathbb{R})$. The continuity of $\mu_{\Delta}$ means that if two sharp values $x$ and $x'$ are very close to each other then, the corresponding random diffusions are very similar, i.e., the probability to get a result in $\Delta$ if the sharp value is $x$ is very close to the probability to get a result in $\Delta$ if the sharp value is $x'$. That is quite common in important  physical applications and  seems to be reasonable from the physical viewpoint. It is then natural to look for general conditions which ensure the continuity of $\lambda\mapsto\mu_{\Delta}$. That is one of the aims of the present work. What we prove is that, in general, the continuity does not hold for all the Borel sets $\Delta$ but only for a ring of subsets which generates the Borel $\sigma$-algebra of the reals. (Anyway, that is sufficient to prove the weak convergence of $\mu_{(\cdot)}(x)$ to $\mu_{(\cdot)}(x')$.) We also prove that the continuity for each Borel set is equivalent to the uniform continuity of $F$ which in its turn is equivalent to require that the smearing in equation (\ref{PositionExample}) can be realized by a strong Feller Markov kernel.

 It is our opinion that the continuity of $\mu_{\Delta}$ over a ring $\mathcal{R}$ which generates the Borel $\sigma$-algebra of the reals could be helpful in dealing with problems connected to the characterization of functions of the kind 
$$G_f(x)=\int f(t)\,d\mu_{t}(x).$$ 
\noindent
A similar (but less general) problem  arises in Ref. \cite{B6} where the relationships between  Naimark extension theorem and the characterization of commutative semispectral measures as smearing of spectral measures are analyzed.
That is a second motivation for the analysis of the continuity properties of $\mu_{\Delta}$. 
%comes form the study of the relationships between Naimark extension theorem and the characterization of commutative semispectral measures as smearing of spectral measures (see \cite{B6}).  Indeed, the main result in Ref. \cite{B6} is based on the proof of the existence of a one-to-one function $f$ such that the function
%$$G_f(i)=\int f(t)\,d\mu_{t}(i),\quad i\in\mathbb{N}$$
%\noindent
%is one-to-one. In order to generalize the main result in Ref. \cite{B6}, it is necessary to prove (under particular conditions) the existence of a one-to-one function $f$ such that the function
%$$G_f(x)=\int f(t)\,d\mu_{t}(x),\quad x\in[0,1]$$
%\noindent
%is one-to-one. It is our opinion that the continuity of $\mu_{\Delta}$ over a ring $\mathcal{R}$ which generates the Borel $\sigma$-algebra of the reals could be helpful in dealing with such a problem.

The results outlined above are contained in the two main theorems of the present work. 

\noindent
The first is a stronger characterization of commutative semispectral measures. In particular, we show (see theorems \ref{weak}) that a semispectral measure is commutative if and only if  there exist a spectral measure $E$ and a  Markov kernel $\mu_{(\cdot)}(\cdot):\Gamma\times\mathcal{B}(\mathbb{R})\to[0,1]$, $\Gamma\subset\sigma(A)$, $E(\Gamma)=\mathbf{1}$, such that 

\begin{equation}\label{Feller}
 F(\Delta)=\int_{\Gamma}\mu_{\Delta}(\lambda)\,dE_{\lambda}
\end{equation}

\noindent
and $\mu_{\Delta}(\cdot)$ is continuous for each $\Delta\in R$ where, $R\subset\mathcal{B}(\mathbb{R})$ is a ring which generates the Borel $\sigma$-algebra of the reals $\mathcal{B}(\mathbb{R})$. It turns out that $\mu_{(\cdot)}(\cdot):\Gamma\times\mathcal{B}(\mathbb{R})\to[0,1]$ is a Feller Markov kernel \cite{Maslowski,Revuz}. Therefore, $F$ is commutative if and only if there exists a Feller Markov kernel $\mu$ such that equation (\ref{Feller}) is satisfied.

\noindent 
We also prove that the family of functions $\{\mu_{\Delta}\}_{\Delta\in\mathcal{B}(\mathbb{R})}$ separates the points of $\sigma(A)$ up to a null set (see theorems \ref{separate}, and \ref{weak}). In other words, the probability measures $\mu_{(\cdot)}(x)$ and $\mu_{(\cdot)}(x')$ which represent the randomizations corresponding to the sharp values $x$ and $x'$ are different. %That is reasonable from the physical viewpoint since the apparatus should be able to distinguish (in the statistical sense) the sharp values $x$ and $x'$ when $x-x'$ is less than the measurement imprecision. 

%the weak Markov kernel $\mu_{(\cdot)}(\cdot)$ can be chosen to be continuous on a ring $\mathcal{R}$ which generates the Borel $\sigma$-algebra of the reals $\mathcal{B}(\mathbb{R})$, i.e.,

The second theorem is a characterization of the semispectral measures which admit a strong Feller Markov kernel, i.e., a Markov kernel $\mu$ such that the function $\lambda\mapsto\mu_{\Delta}(\lambda)$ is continuous for each $\Delta\in\mathcal{B}(\mathcal{R})$. In particular, we prove (see theorem \ref{uni}) that a semispectral measure $F$ admits a strong Feller Markov kernel if and only if it is uniformly continuous. As an example, we develop the details for the unsharp position observable defined in equation (\ref{PositionExample}) above. Finally, we prove (see section 6) that a semispectral measure $F$ which is absolutely continuous with respect to a regular finite measure $\nu$ is uniformly continuous (theorem \ref{abs}). We give some examples of absolutely continuous semispectral measures (see example \ref{PositionExample3}) and analyze the unsharp position observable which is obtained as the marginal of a phase space observable (see section 6.1).
\section{Some preliminaries about Semispectral measures}
\noindent
In what follows, we denote by $\mathcal{B}(\mathbb{R})$ and $\mathcal{B}([0,1])$ the Borel $\sigma$-algebra of $\mathbb{R}$ and [0,1] respectively, by $\textbf{0}$ and $\textbf{1}$ the null and the identity operators, by $\mathcal{L}_s(\mathcal{H})$ the space of all bounded self-adjoint linear operators acting in a Hilbert space $\mathcal{H}$ with scalar product $\langle\cdot,\cdot\rangle$, by $\mathcal{F}(\mathcal{H})=\mathcal{L}_s^+(\mathcal{H})$ the subspace of all positive, bounded self-adjoint operators on $\mathcal{H}$, by $\mathcal{E}(\mathcal{H})\subset\mathcal{F}(\mathcal{H})$ the subspace of all projection operators on $\mathcal{H}$. We use the symbols POVM and PVM to denote semispectral measures and spectral measures respectively.

\begin{definition}
\label{POV}
A Semispectral measure or Positive Operator Valued measure (for short, POVM) is a map $F:\mathcal{B}(\mathbb{R})\to\mathcal{
F}(\mathcal{H})$
such that:
    \begin{equation*}
    F\big(\bigcup_{n=1}^{\infty}\Delta_n\big)=\sum_{n=1}^{\infty}F(\Delta_n).
    \end{equation*}
    \noindent 
 where, $\{\Delta_n\}$ is a countable family of disjoint
    sets in $\mathcal{B}(\mathbb{R})$ and the series converges in the weak operator topology. It is said to be normalized if 
\begin{equation*}   
    F(\mathbb{R})={\bf{1}}
\end{equation*}
\end{definition}    
\begin{definition}
    A POVM is said to be commutative if
    \begin{equation}
    \big[F(\Delta_1),F(\Delta_2)\big]={\bf{0}},\,\,\,\,\forall\,\Delta_1\,,\Delta_2\in\mathcal{B}(\mathbb{R}).
    \end{equation}
    \end{definition}

   \begin{definition}
   A POVM is said to be orthogonal if
    \begin{equation}
    F(\Delta_1)F(\Delta_2)={\bf{0}}\,\,\,\hbox{if}\,\,\Delta_1\cap\Delta_2=
    \emptyset.
    \end{equation}
\end{definition}
\begin{definition}
A Spectral measure or Projection Valued measure (for short, PVM) is an orthogonal, normalized POVM.
\end{definition}
\noindent
It is simple to see that for a PVM $E$, we have $E(\Delta)=E(\Delta)^2$, for any $\Delta \in \mathcal{B}({\mathbb R})$. Then, $E(\Delta)$ is a projection operator for every $\Delta\in\mathcal{B}(\mathbb{R})$, and the PVM is a map $E:\mathcal{B}(\mathbb{R})\to\mathcal{E}(\mathcal{H})$.

\noindent
In quantum mechanics, non-orthogonal normalized POVM are also called \textbf{generalised} or \textbf{unsharp} observables and PVM \textbf{standard} or \textbf{sharp} observables. 

\noindent
In what follows, we shall always refer to real normalized POVM and we shall use the term ``measurable'' for the Borel measurable functions.
For any vector $x\in\mathcal{H}$ the map
$$\langle F(\cdot)x,x\rangle \,:\,\mathcal{B}({\mathbb R})\to {\mathbb R} ,
\qquad
\Delta \mapsto \langle F(\Delta)x,x\rangle,$$
is a Lebesgue-Stieltjes measure. There exists a one-to-one correspondence \cite{Beals} between POV measures $F$ and POV functions $F_{\lambda}:=F((-\infty,\lambda])$. In the following we will use the symbol $d\langle F_{\lambda}x,x\rangle$ to mean integration with respect to the measure $\langle F(\cdot)x,x\rangle$.
\noindent
We shall say that a measurable function $f:N\subset\mathbb{R}\to f(N)\subset\mathbb{R}$ is almost everywhere (a.e.) one-to-one with respect to a POVM $F$ if it is one-to-one on a subset $N'\subset N$ such that $N-N'$ is a null set with respect to $F$.
We shall say that a function $f:{\mathbb R}\to{\mathbb R}$ is bounded with respect to a POVM $F$, if it is equal to a bounded function $g$ a.e. with respect to $F$, that is, if $f=g$ a.e. with respect to the measure $\langle F(\cdot)x,x\rangle$,  $\forall x \in \mathcal{H}$.
\noindent
For any real, bounded and measurable function $f$ and for any POVM $F$, there is a unique \cite{Berberian} bounded self-adjoint operator $B\in\mathcal{L}_s(\mathcal{H})$ such that
\begin{equation}
\label{6}
\langle Bx,x\rangle=\int f(\lambda)d\langle F_{\lambda}x,x\rangle,\quad\text{for each}\quad x\in\mathcal{H}.
\end{equation}
If equation (\ref{6}) is satisfied, we write $B=\int f(\lambda)dF_{\lambda}$ or $B=\int f(\lambda)F(d\lambda)$ equivalently. 
\noindent
\begin{definition}
The spectrum $\sigma(F)$ of a POVM $F$ is the closed set
$$\left\{\lambda\in{\mathbb R}:\,F\big((\lambda-\delta,\lambda+\delta)\big)\neq
0,\,\forall\delta>0,\,\,\right\}.$$
\end{definition}
\noindent
%A POV measure $F$ with a countable spectrum $K\subset\mathbb{R}$ is such that, for any $k\in K$, there exists an operator $F_k$ such that
%$$\sum_k F_k={\bf 1}.$$
%\noindent
%A POV measure $F$ with countable spectrum $K$ can be represented by
%$$F(\Delta)=\sum_{k\in\Delta}F_k.$$
%Thus, in the case of POV measures with countable spectrum, Definition \ref{POV} can be replaced by the following definition:
%\noindent
%\begin{defin}
%\label{finite}
%A POV measure with a countable spectrum $K\subset{\mathbb R}$ is a mapping $F:K\to\{F_k\}_{k\in K}$, also denoted by $\{F_k\}_{k\in K}$, where $\{F_k\}_{k\in K}$ is a set of positive self-adjoint operators acting on a Hilbert space $\cal H$, such that:
%\begin{equation*}
%\sum_{k\in K}F_k=\bf{1}.
%\end{equation*}
%\end{defin}
\noindent
By the spectral theorem \cite{Dunford,Reed}, there is a one-to-one correspondence between PV measures $E$ and self-adjoint operators $B$,
%In fact, we recall the following theorem of functional analysis.
%\begin{teo}{\rm (see Ref. \cite{Reed})}
%There is a one-to-one correspondence between self-adjoint operators $B$ on a Hilbert space $\cal H$ and PV measures $E^B$ on $\cal H$, 
the correspondence being given by
$$B=\int\lambda dE^B_{\lambda}.$$
\noindent
Notice that the spectrum of $E^B$ coincides with the spectrum of the corresponding self-adjoint operator $B$. Moreover, in this case a functional calculus can be developed. Indeed, if $f:{\mathbb R}\to{\mathbb R}$ is a measurable real-valued function, we can define the self-adjoint operator \cite{Reed}
$$f(B)=\int f(\lambda) dE^B_{\lambda}$$
\noindent
where, $E^B$ is the PVM corresponding to $B$. If $f$ is bounded, then $f(B)$ is bounded \cite{Reed}.
\noindent
%\end{teo}

\noindent
In the following we do not distinguish between PVM and the corresponding self-adjoint operators.

%\begin{definition}\label{def8}
%Two bounded self-adjoint operators $A$ and $B$ are said to be equivalent if there exists a bounded, one-to-one, measurable function $f$ such that $A=f(B)$. In this case we write $A\leftrightarrow B$.
%\end{definition}

\noindent
Let $\Lambda$ be a subset of $\mathbb{R}$ and $\mathcal{B}(\Lambda)$ the corresponding Borel $\sigma$-algebra.
%In the following, the symbol $\mu_{(\cdot)}(\lambda)$ denotes a family (with respect to the parameter $\lambda$) of probability measures on $\mathcal{B}(\mathbb{R})$ while the symbol $\omega_{(\cdot)}(\lambda)$ denotes a family of set functions \cite{Loev}. Moreover, we assume that, for every $\Delta\in\mathcal{B}(\mathbb{R})$, the functions $\mu_{(\Delta)}(\lambda)$ and $\omega_{(\Delta)}(\lambda)$ are measurable. Therefore, $\mu_{(\cdot)}(\lambda)$ is a Markov kernel.
\begin{definition}
A real Markov kernel is a map $\mu:\Lambda\times\mathcal{B}(\mathbb{R})\to[0,1]$ such that,
\begin{itemize}
\item[1.] $\mu_{\Delta}(\cdot)$ is a measurable function for each $\Delta\in\mathcal{B}(\mathbb{R})$,
\item[2.] $\mu_{(\cdot)}(\lambda)$ is a probability measure for each $\lambda\in \Lambda$.
\end{itemize}
\end{definition}
\begin{definition}
Let $\nu$ be a measure on $\Lambda$. A map $\mu:\Lambda\times\mathcal{B}(\mathbb{R})\to[0,1]$ is a weak Markov kernel with respect to $\nu$ if:
\begin{itemize}
\item[1.] $\mu_{\Delta}(\cdot)$ is a measurable function for each $\Delta\in\mathcal{B}(\mathbb{R})$,
\item[2.] $0\leq\mu_{\mathbb{R}}(\lambda)\leq 1$,\quad $\nu-a.e.$,
\item[3.]$\mu_{\mathbb{R}}(\lambda)=1$, $\mu_{\emptyset}(\lambda)=0$,\quad $\nu-a.e.$,
\item[4.] for any sequence $\{\Delta_i\}_{i\in\mathbb{N}}$, $\Delta_i\cap\Delta_j=\emptyset$,  
$$\sum_i\mu_{(\Delta_i)}(\lambda)=\mu_{(\cup_i\Delta_i)}(\lambda),\quad \nu-a.e.$$
\end{itemize}
\end{definition}
\begin{definition}
The map $\mu:\Lambda\times\mathcal{B}(\mathbb{R})\to[0,1]$ is a weak Markov kernel with respect to a PVM $E:\mathcal{B}(\Lambda)\to\mathcal{E(H)}$ if it is a weak Markov kernel with respect to each measure $\nu_x(\cdot):=\langle E(\cdot)\,x,x\rangle$, $x\in\mathcal{H}$. 
\end{definition}
\noindent
In the following, by a weak Markov kernel $\mu$ we mean a weak Markov kernel with respect to a PVM $E$. Moreover the function $\lambda\mapsto\mu_{\Delta}(\lambda)$ will be denoted indifferently by $\mu_{\Delta}$ or $\mu_{\Delta}(\cdot)$.    
\begin{definition}
A POV measure $F:\mathcal{B}(\mathbb{R})\to\mathcal{F(H)}$ is said to be a smearing of a POV measure $E:\mathcal{B}(\Lambda)\to\mathcal{E(H)}$ if there exists a weak Markov kernel $\mu:\Lambda\times\mathcal{B}(\mathbb{R})\to[0,1]$ such that,
\begin{equation*}
F(\Delta)=\int_{\Lambda} \mu_{\Delta}(\lambda)d E_{\lambda}, \,\,\,\,\,\,\,\Delta\in\mathcal{B}(\mathbb{R}).
\end{equation*}
\end{definition}

\begin{example}
In the standard formulation of quantum mechanics, the operator
\begin{align*}
Q:L^2(\mathbb{R})&\to L^2(\mathbb{R})\\
\psi(x)\in L^2(\mathbb{R})&\mapsto Q\psi:=x\psi(x)
\end{align*}
\noindent
is used to represent the position observable. A more realistic description of the position observable of a quantum particle is given by a smearing of $Q$ as, for example, the optimal position semispectral measure
\begin{align*}
F^Q(\Delta)&=\frac{1}{l\,\sqrt{2\,\pi}}\int_{-\infty}^{\infty}\Big(\int_{\Delta}e^{-\frac{(x-y)^2}{2\,l^2}}\,d y\Big)\,dE^Q_x=\int_{-\infty}^{\infty}\mu_{\Delta}(x)\,dE^Q_x
\end{align*}
\noindent
where, 
\begin{equation*}
\mu_{\Delta}(x)=\frac{1}{l\,\sqrt{2\,\pi}}\int_{\Delta}e^{-\frac{(x-y)^2}{2\,l^2}}\,d y 
\end{equation*}
\noindent
defines a Markov kernel and $E^Q$ is the spectral measure corresponding to the position operator $Q$.
\end{example}
\noindent
In the following, the symbol $\mu$ is used to denote both  Markov kernels and weak Markov kernels. The symbols $A$ and $B$ are used to denote self-adjoint operators. 
\begin{definition}
Whenever $F$, $A$, and $\mu$ are such that $F(\Delta)=\mu_{\Delta}(A)$, $\Delta\in\mathcal{B}(\mathbb{R})$, we say that $(F,A,\mu)$ is a von Neumann triplet. 
\end{definition}
\noindent
The following theorem establishes a relationship between commutative semispectral measures and spectral measures and gives a characterization of the former. Other characterizations and an analysis of the relationships between them can be found in Ref.s \cite{Ali,Holevo,Ali5,P1}.
\begin{theorem}[\cite{B1,P}]\label{Cha}
 A semispectral measure $F$ is commutative if and only if there exist a bounded self-adjoint operator $A$ and a Markov kernel (weak Markov kernel) $\mu$ such that $(F,A,\mu)$ is a von Neumann triplet. 
\end{theorem}
\begin{corollary}\label{smearing}
 A semispectral measure $F$ is commutative if and only if it is a smearing of a PV measure $E$ with bounded spectrum.
\end{corollary}
\begin{definition}
The von Neumann algebra generated by the semispectral measure $F$ is the von Neumann algebra generated by the set $\{F(\Delta),\,\Delta\in\mathcal{B}(\mathbb{R})\}$. 
\end{definition}
\begin{definition}
If $A$ and $F$  in theorem \ref{Cha} generate the same von Neumann algebra then $A$ is named the sharp version of $F$.
\end{definition}
\begin{theorem}\label{unique}\cite{B1}
The sharp version $A$ is unique up to almost everywhere bijections. 
\end{theorem}

\section{On the separation properties of $\mu$}
\noindent
In the following, the symbol $\mathcal{S}$ denotes the family of open intervals in $\mathbb{R}$ with rational end-points. The symbol $\mathcal{R(S)}$ denotes the ring generated by $\mathcal{S}$. Notice that $\mathcal{S}$ is countable. Then, by theorem c, page 24, in Ref. \cite{Halmos}, $\mathcal{R(S)}$ is countable too. Moreover,  $\mathcal{R(S)}$ generates the Borel $\sigma$-algebra $\mathcal{B}(\mathbb{R})$.

A weak Markov kernel $\mu$ such that $(F,A,\mu)$ is a von Neumann triplet, separates the point of $\Gamma\subset\sigma(A)$ if the family of functions $\{\mu_{\Delta}\}_{\Delta\in\mathcal{B}(\mathbb{R})}$ separates the points of $\Gamma$ or, in other words, if the set functions $\{\mu_{(\cdot)}(\lambda)\}_{\lambda\in\Gamma}$ are distinct. It is then natural to ask if in general $\mu$ has that property. The following theorem answers in the positive.
\begin{theorem}\label{separate}
Let $(F,A,\mu)$ be a von Neumann triplet and suppose that $A$ is a sharp version of $F$. Then, there exists a set $\Gamma\subseteq\sigma(A)$, $E^A(\Gamma)=\mathbf{1}$, such that the family of functions $\{\mu_{\Delta}(\cdot)\}_{\Delta\in\mathcal{B}(\mathbb{R})}$ separates the points of $\Gamma$.
%there exists a countable family $\mathcal{G}\subset\mathcal{B}(\mathbb{R})$ containing $\mathcal{R(S)}$, and a set $N\subset\sigma(A)$ such that, the family of functions  $\{\mu_{\Delta}(\cdot)\}_{\Delta\in\mathcal{G}}$ distinguishes the points of $N$. 
\end{theorem} 
\begin{proof}
In the following, $\mathcal{A}^W(F)$ denotes the von Neumann algebra generated by $\{F(\Delta)\}_{\Delta\in\mathcal{B}(\mathbb{R})}$, $O_2:=\{F(\Delta)\}_{\Delta\in\mathcal{R(S)}}$ and $\mathcal{A}^C(O_2)$ is the $C^*$-algebra generated by $O_2$. The von Neumann algebra generated by $\mathcal{A}^C(O_2)$ coincides with $\mathcal{A}^W(F)$ (see appendix A). Moreover, $\mathcal{A}^W(F)=\mathcal{A}^W(A)$ since $A$ is the sharp version of $F$ and generates $\mathcal{A}^W(F)$.
By the Gelfand-Naimark theorem \cite{Dunford,Naimark}, there is a * isomorphism $\phi$ between $\mathcal{A}^C(O_2)$ and the algebra of continuous functions $\mathcal{C}(\Lambda_2)$ where $\Lambda_2$ is the spectrum of $\mathcal{A}^C(O_2)$. Moreover,
\begin{equation*}
f\in \mathcal{C}(\Lambda_2)\mapsto\phi(f)=\int_{\Lambda_2} f(\lambda)\,d\widetilde{E}_{\lambda}
\end{equation*}
\noindent
where, $\widetilde{E}$ is the spectral measure from the Borel $\sigma$ algebra $\mathcal{B}(\Lambda_2)$ to $\mathcal{E(H)}$ whose existence is assured by theorem 1, page 895, in Ref. \cite{Dunford}. The Gelfand-Naimark isomorphism $\phi$ can be extended to a homomorphism between the algebra of the Borel functions on $\Lambda_2$ and the von Neumann algebra $\mathcal{A}^W(F)=\mathcal{A}^W(A)$ generated by $\mathcal{A}^C(O_2)$ (see Ref. \cite{Dixmier1}, page 360, section 3). Therefore, there is a Borel function $h$ such that 
\begin{equation}
A=\int_{\Lambda_2} h(\lambda)\,d\widetilde{E}_{\lambda}
\end{equation} 
\noindent
Let $\{\Delta_i\}_{i\in\mathbb{N}}$ denote an enumeration of the set $\mathcal{R(S)}$. Since $\mathcal{A}^C(O_2)$ is the smallest uniform closed algebra containing $\{F(\Delta_i)\}_{i\in\mathbb{N}}$, $\mathcal{C}(\Lambda_2)$ is the smallest uniform closed algebra of functions containing $\{\nu_{\Delta_i}:=\phi^{-1}(F(\Delta_i))\}_{i\in\mathbb{N}}$. In other words $\{\nu_{\Delta_i}\}_{i\in\mathbb{N}}$ generates $\mathcal{C}(\Lambda_2)$. The Stone-Weierstrass theorem \cite{Dunford} assures that $\{\nu_{\Delta_i}\}_{i\in\mathbb{N}}$ separates the points in $\Lambda_2$. 

\noindent
On the other hand, the fact that  $(F,A,\mu)$ is a von Neumann triplet, implies that, for each $\Delta_i\in\mathcal{R(S)}$, there is a Borel function $\mu_{\Delta_i}$ such that 
\begin{equation*}
\int_{\Lambda_2}\nu_{\Delta_i}(\lambda)\,d\widetilde{E}_{\lambda}=F(\Delta_i)=\mu_{\Delta_i}(A)=\int_{\Lambda_2}\mu_{\Delta_i}(h(\lambda))\,d\widetilde{E}_{\lambda}.
\end{equation*}
\noindent
Then, for each $\Delta_i\in\mathcal{R(S)}$, there is a set $M_i\subset\Lambda_2$, $\widetilde{E}(M_i)=\mathbf{1}$, such that
\begin{equation}\label{a}
\mu_{\Delta_i}(h(\lambda))=\nu_{\Delta_i}(\lambda),\quad \lambda\in M_i.
\end{equation}
\noindent
Let $M:=\cap_{i=1}^{\infty} M_i$. Then, 
\begin{equation*}
\widetilde{E}(M)=\lim_{n\to\infty}\widetilde{E}(\cap_{i=1}^{n} M_i)=\lim_{n\to\infty}\prod_{i=1}^{n}\widetilde{E}(M_i)=\mathbf{1}
\end{equation*}
\noindent
and, for each $i\in\mathbb{N}$,
\begin{align}\label{2}
(\mu_{\Delta_i}\circ h)(\lambda)=\nu_{\Delta_i}(\lambda),\quad\lambda\in M\subseteq \Lambda_2. 
\end{align}
\noindent
Since $\{\nu_{\Delta_i}\}_{i\in\mathbb{N}}$ separates the points in $\Lambda_2$, it separates the points in $M$.  Then, equation (\ref{2}) implies that $\{\mu_{\Delta_i}\}_{i\in\mathbb{N}}$ separates the points in $\Gamma:=h(M)$. 
%Moreover, 
%\begin{equation*}
%E(V)=E(M\cap(\sigma(A)-I))=E(M)E(\sigma(A)-I)=\mathbf{1}.
%\end{equation*} 
\noindent
Moreover\footnote{ Notice that $h(M)$ is a Borel set. In order to prove that, we first recall that $\Lambda_2$ is a Polish space (that is, a complete, separable, space \cite{Kuratowski}). Indeed, by theorem 11, page 871, in Ref. \cite{Dunford}, it is homeomorphic to a closed subspace of the Cartesian product $\prod_{i=1}^{\infty}\sigma(F(\Delta_i))$, where $\sigma(F(\Delta_i))$ is a complete separable metric space, and by theorem 2, page 406, and theorem 6, page 156, in Ref. \cite{Kuratowski1}, it is complete and separable. Moreover, $h$ is measurable and injective on $M$. Therefore, Soulsin's theorem (see theorem 9 page 440 and Corollary 1 page 442 in Ref. \cite{Kuratowski}) assures that $h(M)$ is a Borel set.},
\begin{align*}
E^A(\Gamma)=E^A(h(M))=\widetilde{E}[h^{-1}(h(M))]=\mathbf{1}
\end{align*}
\noindent
where, $E^A$ is the spectral measure defined by the relation
$$E^A(\Delta)=\widetilde{E}(h^{-1}(\Delta))$$
\noindent
and such that,
$$A=\int x\,dE^A_x$$
\noindent
while, $h^{-1}(h(M))$ is a Borel set containing $M$.

\noindent
We have proved that the set of functions $\{\mu_{\Delta_i}\}_{i\in\mathbb{N}}$ separates the points of $\Gamma$ and that $E^A(\Gamma)=\mathbf{1}$. In other words,
$$\mu_{(\cdot)}(\lambda)\neq\mu_{(\cdot)}(\lambda'),\quad \lambda\neq\lambda',\,\,\,\,\lambda,\lambda'\in \Gamma.$$ 
\end{proof}

\section{Characterization of Commutative Semi-spectral Measures by means of Feller Markov kernels}
As we have seen in the last section, theorem \ref{Cha} asserts that a semispectral measure $F$ is commutative if and only if there exist a bounded self-adjoint operator $A$ and a weak Markov kernel (Markov kernel) $\mu$ such that $F(\Delta)=\mu_{\Delta}(A)$.  In the present section we study the continuity of the functions $\mu_{\Delta}$. In particular, %we prove (see theorem \ref{weak} below) that, if $F$ is commutative, there exists a weak Markov kernel $\mu$ such that: a) $(F,A,\mu)$ is a von Neumann triplet, b) $\mu_{(\cdot)}(\lambda)$, $\lambda\in\sigma(A)$ is additive on a ring $\mathcal{R(S)}$ which generates the Borel $\sigma$-algebra of  the reals and c) $\mu_{\Delta}$ is continuous for each $\Delta\in \mathcal{R(S)}$. .
\noindent
we introduce the concept of strong Markov kernel, i.e., a weak Markov kernel $\mu_{(\cdot)}(\cdot):\Lambda\times\mathcal{B}(\mathbb{R})\to[0,1]$ with respect to a PVM $E:\mathcal{B}(\Lambda)\to\mathcal{E(H)}$ such that $\mu_{(\cdot)}(\lambda)$ is a probability measure for each $\lambda\in \Gamma\subset \Lambda$, $E(\Gamma)=\mathbf{1}$. Then, we prove (theorem \ref{weak}) that in order to realize the smearing in corollary \ref{smearing}, one can use a strong Markov kernel $\mu$ such that $\mu_{\Delta}$ is continuous for each $\Delta\in R$, where $R$ is a ring which generates the Borel $\sigma$-algebra of the reals.  It is worth remarking that $\mu_{(\cdot)}(\cdot):\Gamma\times\mathcal{B}(\mathbb{R})\to[0,1]$ is a Feller Markov kernel. Therefore, $F$ is commutative if and only if there exists  a bounded self-adjoint operator $A$ and a Feller Markov kernel $\mu$ such that 
$$F(\Delta)=\int_{\Gamma}\mu_{\Delta}(\lambda)\,dE_{\lambda}.$$
 Moreover, the family of functions $\{\mu_{\Delta}\}_{\Delta\in R}$ separates the points in $\Gamma$ (see theorems \ref{separate} and \ref{weak}).
%Moreover, it is worth remarking that the choice of the ring $\mathcal{R}$ is quite arbitrary, the only condition it must satisfy is that the $\sigma$-algebra it generates coincides with $\mathcal{B}(R)$. 

% The following semi-ring of subsets of $[0,1]$ will be used:
% $$\mathcal{S}=\left\{[0,{1/{2^{n-1}}}],\,({k/{2^{n-1}}},{k+1/{2^{n-1}}}]\,\,\vert\,\,, k=1,2,\ldots, 2^{n-1}-1,\,n\in \mathbb{N}\right\}.$$
 In order to prove the main theorem we need the following definitions.
\begin{definition}
Let $E:\mathcal{B}(\Lambda)\to\mathcal{E(H)}$ be a PVM. The map $\mu_{(\cdot)}(\cdot):\Lambda\times\mathcal{B}(\mathbb{R})\to[0,1]$ is a strong Markov kernel with respect to $E$ if it is a weak Markov kernel and there exists a set $\Gamma\subset\Lambda$, $E(\Gamma)=\mathbf{1}$, such that $\mu_{(\cdot)}(\cdot):\Gamma\times\mathcal{B}(\mathbb{R})\to[0,1]$ is a Markov kernel with respect to $E$. A strong Markov kernel is denoted by the symbol $(\mu, E,\Gamma\subset\Lambda)$.
\end{definition} 
 \begin{definition}
A Feller Markov kernel is a  Markov kernel $\mu_{(\cdot)}(\cdot):\Lambda\times\mathcal{B}(\mathbb{R})\to[0,1]$ such that the function 
$$G(\lambda)=\int_{\Lambda}f(t)\,d\mu_t(\lambda),\quad\lambda\in\Lambda$$
\noindent
is continuous and bounded whenever $f$ is continuous and bounded. 
\end{definition}
\begin{theorem}\label{weak}
A real POVM $F:\mathcal{B}(\mathbb{R})\to\mathcal{F(H)}$ is commutative if and only if, there exists a bounded self-adjoint operator $A=\int \lambda\,dE_\lambda$ with spectrum $\sigma(A)\subset[0,1]$ and a strong Markov Kernel  $(\mu, E, \Gamma\subset\sigma(A))$ 
%$\mu:\Gamma\subset\sigma(A)\times\mathcal{B}(\mathbb{R})\to[0,1]$,  
such that:
\begin{enumerate}
\item[1)] $\mu_{\Delta}(\cdot):\sigma(A)\to[0,1]$ is continuous for each $\Delta\in\mathcal{R(S)}$,
%\item[2)] $E(\Gamma)=\mathbf{1}$,
\item[2)] $F(\Delta)=\int_{\Gamma}\mu_{\Delta}(\lambda)\,dE_{\lambda},\quad\Delta\in\mathcal{B}(\mathbb{R})$.
\item[3)] $\mu$ separates the points in $\Gamma$.
\end{enumerate}
\noindent
Moreover, $\mu:\Gamma\times\mathcal{B}(\mathbb{R})\to[0,1]$ is a Feller Markov kernel.

\end{theorem}
\begin{proof}
Let $\mathcal{A}^W(F)$ be the von Neumann algebra generated by $F$. $\mathcal{A}^W(F)$ coincides with the von Neumann algebra generated by $\{F(\Delta)\}_{\Delta\in\mathcal{R(S)}}$ where, $\mathcal{R(S)}\subset\mathcal{B}(\mathbb{R})$ is the ring generated by the family $\mathcal{S}$ of open intervals with rational end-points (see appendix A for the proof). We recall that both $\mathcal{S}$ and $\mathcal{R(S)}$  are countable (see theorem c, page 24, in Ref. \cite{Halmos}). 

Now, we proceed to the proof of the existence of $A$. Let $\{\Delta_i\}_{i\in\mathbb{N}}$ be an enumeration of the set $\mathcal{R(S)}$ and $O_2:=\{F(\Delta)\}_{\Delta\in\mathcal{R(S)}}$ . Let $E^{(i)}$ denote the spectral measure corresponding to $F(\Delta_i)\in O_2$. We have $F(\Delta_i)=\int x\,d E^{(i)}_x$. Therefore, for each $i,k\in\mathbb{N}$ there exists a division $\{\Delta_j^{(i,k)}\}_{j=1,\dots,m_{i,k}}$ of $[0,1]$ such that 
\begin{equation}\label{D}
\big\|\sum_{j=1}^{m_{i,k}}x^{(i,k)}_j\,E^{(i)}(\Delta_j^{(i,k)})-F(\Delta_i)\big\|\leq \frac{1}{k}.
\end{equation} 

\noindent
By the spectral theorem \cite{Dunford} the von Neumann algebra $\mathcal{A}^W(F)$ contains all the projection operators in the spectral resolution of $F(\Delta)$, $\Delta\in\mathcal{B}(\mathbb{R})$. Therefore, the von Neumann algebra $\mathcal{A}^W(D)$ generated by the set $D:=\{E^{(i)}(\Delta_j^{i,k}),\,j\leq m_{i,k},\,i,k\in\mathbb{N}\}$ is contained in $\mathcal{A}^W(F)$ and then 
\begin{equation}\label{D1}
\mathcal{A}^W(D)\subset\mathcal{A}^W(F)=\mathcal{A}^W(O_2).
\end{equation}
\noindent
 Moreover, the $C^*$-algebra $\mathcal{A}^C(D)$ generated by $D$ contains the $C^*$-algebra $\mathcal{A}^C(O_2)$ generated by $O_2$ (see equation (\ref{D})). Summing up the preceding observations, we have 
\begin{equation*}
\mathcal{A}^C(O_2)\subset\mathcal{A}^C(D)\subset\mathcal{A}^W(F).
\end{equation*}
\noindent
By the double commutant theorem \cite{Kadison}, 
\begin{align*}
\mathcal{A}^W(F)&=[\mathcal{A}^C(O_2)]''\subset[\mathcal{A}^C(D)]''=\mathcal{A}^W(D)
%\mathcal{A}^W(D)&=[\mathcal{A}^C(D)]''\subset[\mathcal{A}^W(F)]''=[\mathcal{A}^W(F)]
\end{align*} 
\noindent
so that (see equation \ref{D1}),
\begin{equation}\label{double}
\mathcal{A}^W(D)=\mathcal{A}^W(F).
\end{equation} 

\noindent
By theorem 11, page 871 in Ref. \cite{Dunford}, the spectrum $\Lambda$ of $\mathcal{A}^C(D)$ is homeomorphic to a closed subset of $\prod_{i=1}^{\infty}\{0,1\}$. Let $\pi:\Lambda\to\prod_{i=1}^{\infty}\{0,1\}$ denote the homeomorphism between the two spaces.

\noindent
Now, if we identify $\Lambda$ with a closed subset of $\prod_{i=1}^{\infty}\{0,1\}$, we can prove the existence of a continuous function distinguishing the points of $\Lambda$. Indeed, let $\pi(\lambda)=\bar{x}:=(x_1,\dots,x_n,\dots)\in\prod_{i=1}^\infty\{0,1\}$. The function
$$f(\lambda)=\sum_{i=1}^{\infty}\frac{x_i}{3^i}$$
is continuous and injective and then it distinguishes the points of $\Lambda$. Moreover, since $\Lambda$ and $[0,1]$ are Hausdorff,  the map $f:\Lambda\to f(\Lambda)$ is a homeomorphism.

\noindent
By theorem 1, page 895, in Ref. \cite{Dunford}, there exists a spectral measure $\widetilde{E}:\mathcal{B}(\Lambda)\to\mathcal{F(H)}$ such that the map 
\begin{align}\label{E}
T:\mathcal{C}(\Lambda)&\to B(\mathcal{H})\\
g&\mapsto T(g)=\int_{\Lambda}g(\lambda) d\widetilde{E}_{\lambda}\notag
\end{align}
defines an isometric $^*$-isomorphism between $\mathcal{A}^C(D)$ and $\mathcal{C}(\Lambda)$.

\noindent
The fact that $f$ distinguishes the points of $\Lambda$, implies that the self-adjoint operator
\begin{equation*}
A=\int_{\Lambda} f(\lambda)\,d\widetilde{E}_{\lambda}
\end{equation*}
is a generator of the von Neumann algebra $\mathcal{A}^W(D)=\mathcal{A}^W(F)$. %Then, by equation (\ref{double}), $A$ is the generator of $\mathcal{A}^W(F)$.
Indeed, by the Stone-Weierstrass theorem, %$\widetilde{f}\in\mathcal{C}(\Lambda)$  implies that 
$\mathcal{C}(\Lambda)$ is singly generated, in particular $f$ is a generator. Then, the isomorphism between $\mathcal{A}^C(D)$ and $\mathcal{C}(\Lambda)$ assures that $\mathcal{A}^C(D)$ is singly generated and that $A$ is a generator. Hence,  $\mathcal{A}^W(F)=\mathcal{A}^W(D)=[ \mathcal{A}^C(D)]''$ is singly generated.  In particular, $A$ generates $\mathcal{A}^W(F)$, i.e., $\mathcal{A}^W(F)=\mathcal{A}^W(A)$.

\vskip.2cm
Now, we proceed to the proof of the existence of the weak Markov kernel $\widetilde{\nu}$ such that $(F,A,\widetilde{\nu})$ is a von Neumann triplet.

By (\ref{E}), for each $\Delta\in\mathcal{R(S)}$, there exists a continuous function $\gamma_{\Delta}\in\mathcal{C}(\Lambda)$ such that 
$$F(\Delta)=\int_{\Lambda}\gamma_{\Delta}(\lambda)\, d\widetilde{E}_{\lambda}.$$

%\noindent
%Hence, by theorem ?, page ? \cite{Dunford},
%\begin{equation*}
%\widetilde{\mu}_{\Delta}(\lambda)=(\omega_{\Delta}(\lambda)\circ f)(\lambda)\quad \widetilde{E}-a.e.
%\end{equation*} 

\noindent
Now, we show that, for each $\Delta\in\mathcal{R(S)}$, there is a continuous function $\nu_{\Delta}:\sigma(A)\to[0,1]$ from the spectrum of $A$ to the interval $[0,1]$ such that $\nu_{\Delta}(f(\lambda))=\gamma_{\Delta}(\lambda)$, $\lambda\in\Lambda$, and $F(\Delta)=\nu_{\Delta}(A)$.  

\noindent
To prove this, let us consider the function
\begin{equation*}
\nu_{\Delta}(t):=(\gamma_{\Delta}\circ f^{-1})(t),\quad\Delta\in\mathcal{R(S)}.
%\begin{cases}
%\omega_{\Delta}(t) & if \quad\omega_{\Delta}(t)=\widetilde{\mu}_{\Delta}(\lambda), \quad\lambda=f^{-1}_A(t)\\
%\widetilde{\mu}_{\Delta}(\lambda) & if \quad\omega_{\Delta}(t)\neq\widetilde{\mu}_{\Delta}(\lambda),\quad\lambda= f^{-1}_A(t)
%\end{cases}
\end{equation*}
\noindent
It is continuous since it is the composition of continuous functions and, 
$$\nu_{\Delta}(f(\lambda))=\gamma_{\Delta}(f^{-1}(f(\lambda)))=\gamma_{\Delta}(\lambda).$$ 

\noindent
Moreover, 
$$\nu_{\Delta}(A)=F(\Delta),\quad\forall\Delta\in\mathcal{R(S)}.$$
\noindent
Indeed, by the change of measure principle (page 894, ref. \cite{Dunford}), 
\begin{align*}
F(\Delta)&=\int_{\Lambda}\gamma_{\Delta}(\lambda)\,d\widetilde{E}_{\lambda}=\int_{\Lambda}\gamma_{\Delta}(f^{-1}(f(\lambda)))\,d\widetilde{E}_{\lambda}\\
&=\int_{\sigma(A)}\gamma_{\Delta}(f^{-1}(t))\,d E_{t}=\int_{\sigma(A)}\nu_{\Delta}(t)\,d E_{t}=\nu_{\Delta}(A)
\end{align*} 
where $\sigma(A)=f(\Lambda)$ is the spectrum of $A$ and $E$ is the spectral measure corresponding to $A$ defined by the relation $E(\Delta)=\widetilde{E}(f^{-1}(\Delta))$, $\Delta\in\mathcal{B}(\sigma(A))$ (see corollary 10, page 902, in Ref. \cite{Dunford}).  

\noindent
For each $\lambda\in\sigma(A)$, the map $\nu_{(\cdot)}(\lambda):\mathcal{R(S)}\to[0,1]$ defines an additive set function. Indeed, let $\Delta\in\mathcal{R(S)}$ be the disjoint union of the sets $\Delta_1,\Delta_2\in\mathcal{R(S)}$. Then,
\begin{align*}
\int\nu_{(\Delta_1\cup\Delta_2)}(\lambda)\,dE_{\lambda}&=F(\Delta_1\cup\Delta_2)=F(\Delta_1)+F(\Delta_1)\\
&=\int\nu_{\Delta_1}(\lambda)\,dE_{\lambda}+\int\nu_{\Delta_2}(\lambda)\,dE_{\lambda}\\
&=\int\big[\nu_{\Delta_1}(\lambda)+\nu_{\Delta_2}(\lambda)\big]\,dE_{\lambda}
\end{align*}
so that, by the continuity of the functions $\nu_{(\Delta_1)}(\lambda)$ and $\nu_{(\Delta_2)}(\lambda)$, we get (see theorem 1, page 895, in Ref. \cite{Dunford}) 
\begin{equation*}
\nu_{(\Delta_1)}(\lambda)+\nu_{(\Delta_2)}(\lambda)=\nu_{(\Delta_1\cup\Delta_2)}(\lambda),\quad\forall\lambda\in\sigma(A).
\end{equation*}
%Therefore, $\nu$ can be extended to an additive set function on the ring $\mathcal{R(S)}$ generated by $\mathcal{S}$. We will use the same symbol both for $\nu$ and for its extension to $\mathcal{R(S)}$. It is worth noticing that $\nu_{\Delta}$ is continuous for each $\Delta\in\mathcal{R(S)}$. Indeed (see the proof of theorem c, page 23 in Ref. \cite{Halmos}), each $\Delta\in\mathcal{R(S)}$ is obtained  by means of a finite number of set operations (union, difference) on sets in $\mathcal{S}$.  Consequently, $\nu_{\Delta}$ is obtained by a finite number of operations (sum and difference) on continuous functions in $\{\nu_{\Delta}\}_{\Delta\in\mathcal{S}}$.

Now, we extend $\nu$ to all the Borel $\sigma$-algebra of $[0,1]$.

\noindent
Since  $A$ is the generator of $\mathcal{A}^W(F)$, for each $\Delta\in\mathcal{B}([0,1])$, there exists a Borel function $\omega_{\Delta}$ such that.
\begin{equation*}
F(\Delta)=\int_{\sigma(A)}\omega_{\Delta}(t)\,d E_t=\int_{\Lambda} (\omega_{\Delta}\circ f)(\lambda)\,d\widetilde{E}_{\lambda} 
\end{equation*}
\noindent
Then, we can consider the map $\widetilde{\nu}:\sigma(A)\times\mathcal{B}(\mathbb{R})\to[0,1]$ defined as follows
\begin{equation*}
\widetilde{\nu}_{\Delta}(\lambda)=
\begin{cases}
\nu_{\Delta}(\lambda) & if \quad \Delta\in\mathcal{R(S)}\\
\omega_{\Delta}(\lambda) & if \quad \Delta\notin\mathcal{R(S)}.
\end{cases}
\end{equation*}
\noindent
%%%%%%%%%%%%%%%%%%%%%%%%%%%%%%%%%%%%%%%%%%%%%%%%%%%%%%
%In order to prove the $\sigma$-additivity of $\mu$ on $\mathcal{R(S)}$, let us consider a set $\Delta\in\mathcal{R(S)}$ such that $\Delta=\cup_{i=1}^{\infty}\Delta_i$, $\Delta_i\in\mathcal{R(S)}$. We have,
%\begin{align*}
%\int\mu_{\Delta}(\lambda)\,dE_{\lambda}&=\int\mu_{(\cup_{i=1}^{\infty}\Delta_i)}(\lambda)\,dE_{\lambda}=F(\cup_{i=1}^{\infty}\Delta_i)\\
%&=\sum_{i=1}^{\infty}F(\Delta_i)=\sum_{i=1}^{\infty}\int\mu_{\Delta_i}(\lambda)\,dE_{\lambda}=\int\sum_{i=1}^{\infty}\mu_{\Delta_i}(\lambda)\,dE_{\lambda}
%\end{align*}
%so that, by the continuity of the functions $\mu_{\Delta_i}(\lambda)$ and $\mu_{\Delta}(\lambda)$, we get (see theorem 1, page 895, in Ref. \cite{Dunford})
%\begin{equation*}
%\mu_{(\Delta_1)}(\lambda)+\mu_{(\Delta_2)}(\lambda)=\mu_{(\Delta_1\cup\Delta_2)}(\lambda),\quad\forall\lambda\in\sigma(A).
%\end{equation*}
%%%%%%%%%%%%%%%%%%%%%%%%%%%%%%%%%%%%%%%%%%%%%%%%%%%%%%
\noindent
Since $\widetilde{\nu}$ coincides with $\nu$ on $\mathcal{R(S)}$ it is additive on $\mathcal{R(S)}$. 

\noindent

\noindent
In order to prove that $\widetilde{\nu}$ is a weak Markov kernel, let us consider a set $\Delta\in\mathcal{B}(\mathbb{R})$ which is the disjoint union of the sets $\{\Delta_i\}_{i\in\mathbb{N}}$, $\Delta_i\in\mathcal{B}(\mathbb{R})$. Then,
\begin{align*}
\int\widetilde{\nu}_{(\cup_{i=1}^{\infty}\Delta_i)}(x)\,dE_{x}&=\int\widetilde{\nu}_{\Delta}(x)d E_x=F(\Delta)=\sum_{i=1}^{\infty}F(\Delta_i)\\
&=\sum_{i=1}^{\infty}\int\widetilde{\nu}_{\Delta_i}(x)\,dE_x=\int\sum_{i=1}^{\infty}\widetilde{\nu}_{\Delta_i}(x)\,dE_x
\end{align*}
so that, by Corollary 9, page 900, in Ref. \cite{Dunford},
\begin{equation*}
\sum_{i=1}^{\infty}\widetilde{\nu}_{\Delta_i}(x)=\widetilde{\nu}_{\Delta}(x), \quad E-a.e,
\end{equation*}
which implies that $\widetilde{\nu}:\sigma(A)\times\mathcal{B}(\mathbb{R})\to[0,1]$ is a weak Markov kernel. In particular $(F,A,\widetilde{\nu})$ is a von Neumann triplet.

Now, we proceed to prove the existence of the Markov kernel $\mu:\Gamma\times\mathcal{B}(\mathbb{R})\to[0,1]$ such that items 1, 2, and 3 of the theorem are satisfied. 

By corollary 1 in Ref. \cite{B1}, starting from $\widetilde{\nu}:\sigma(A)\times\mathcal{R(S)}\to[0,1]$ it is possible to define a Markov kernel $\omega:\sigma(A)\times\mathcal{B}(\mathbb{R})\to[0,1]$ such that $(F,A,\omega)$ is a von Neumann triplet. Since $(F,A,\widetilde{\nu})$ is a von Neumann triplet, for each $\Delta\in\mathcal{B}(\mathbb{R})$,
\begin{equation*}
\int \widetilde{\nu}_{\Delta}(\lambda)\,dE_\lambda=F(\Delta)=\int \omega_{\Delta}(\lambda)\,dE_\lambda
\end{equation*}
\noindent
hence,
\begin{equation}\label{ae}
\omega_{\Delta}(\lambda)=\widetilde{\nu}_{\Delta}(\lambda),\quad E-a.e.
\end{equation}
\noindent
%In particular, for $\Delta\in\mathcal{R(S)}$, $\nu_{\Delta}$ is $E-a.e.$ continuous since $\mu_{\Delta}$ is continuous.
Now, let $\{\Delta_i\}_{i\in\mathbb{N}}$ be an enumeration of $\mathcal{R(S)}$. By equation (\ref{ae}), for each $i\in\mathbb{N}$, there is a set $N_i\subset\sigma(A)$, $E(N_i)=\mathbf{0}$, such that 
\begin{equation}
\omega_{\Delta_i}(\lambda)=\widetilde{\nu}_{\Delta_i}(\lambda),\quad \lambda\in \sigma(A)-N_i.
\end{equation}
\noindent
Then, for each $i\in\mathbb{N}$,
\begin{equation}
\omega_{\Delta_i}(\lambda)=\widetilde{\nu}_{\Delta_i}(\lambda),\quad\lambda\in \sigma(A)-N
\end{equation}
\noindent
where, 
\begin{equation*}
N:=\cup_{i=1}^{\infty}N_i,\quad E(N)=\mathbf{0}.
\end{equation*}
\noindent
Therefore, for almost all $\lambda\in\sigma(A)$, $\widetilde{\nu}_{(\cdot)}(\lambda)$ is $\sigma$-additive on $\mathcal{R(S)}$.

%\noindent
%By proposition 2, section IV, in Ref. \cite{B1}, for each $\lambda\in[0,1]-N$, 
%$$\mu_{\Delta}(\lambda)=\nu_{\Delta}(\lambda),\quad\Delta\in\mathcal{R(S)}.$$
\noindent
Now, we can define the map
\begin{equation*}
\mu_{(\cdot)}(\lambda)=
\begin{cases}
\widetilde{\nu}_{(\cdot)}(\lambda) & \lambda\in N\\
\omega_{(\cdot)}(\lambda) & \lambda\in \sigma(A)-N
\end{cases}
\end{equation*}
\noindent
If we put $\Gamma=\sigma(A)-N$, we have that $\mu_{(\cdot)}(\cdot):\Gamma\times\mathcal{B}(\mathbb{R})\to[0,1]$ is a Markov kernel. Therefore, $\mu_{(\cdot)}(\cdot):\sigma(A)\times\mathcal{B}(\mathbb{R})\to[0,1]$ is a strong Markov kernel.

\noindent
Notice that, for each $\Delta\in\mathcal{R(S)}$ and $\lambda\in\sigma(A)$,
$$\mu_{\Delta}(\lambda)=\widetilde{\nu}_{\Delta}(\lambda)$$
\noindent
so that, $\mu_{\Delta}$ is continuous for each $\Delta\in\mathcal{R(S)}$ and additive on $\mathcal{R(S)}$. We also have,
$$\mu_{\Delta}(A)=\omega_{\Delta}(A)=F(\Delta), \quad\Delta\in\mathcal{R(S)}.$$
We have proved items 1, 2, and 3. Item 4 comes from theorem \ref{separate}. 

It remains to prove that $\mu$ is a Feller Markov kernel. By item 1, $\mu_{\Delta}$ is continuous for each $\Delta\in \mathcal{R(S)}$. Notice that for each open set $O\in\mathcal{B}(\mathbb{R})$, there is a countable family of sets $\Delta_i\in\mathcal{R(S)}$ such that $O=\cup_{i=1}^{\infty}\Delta_i$. Therefore, by theorem 2.2 in Ref. \cite{Billingsley}, $\mu_{(\cdot)}(\lambda_n)$ converges weakly to $\mu_{(\cdot)}(\lambda)$, i.e.,
$$\lim_{n\to\infty}\int f(t)\,\mu_t(\lambda_n)=\int f(t)\,\mu_t(\lambda),\quad f\in\mathcal{C}_b(\mathbb{R})$$
\noindent
whenever $\lim_{n\to\infty}\lambda_n=\lambda$ and $\mathcal{C}_b(\mathbb{R})$ is the space of bounded, continuous functions.

Since $F(\Delta)=\mu_{\Delta}(A)$ implies the commutativity of $F$, the theorem is proved.
\end{proof}
%\begin{remark} 
%Since the function $f:\Lambda\to[0,1]$ introduced in the proof of theorem \ref{weak}, is one-to-one and continuous, its restriction $\widetilde{f}$ to the set $\Lambda_2$, is one-to-one and continuous. Therefore, $\widetilde{f}\in\mathcal{C}(\Lambda_2)$ and $\widetilde{f}$ separates the points of $\Lambda_2$. That assures that $\mathcal{A}^C(O_2)$ is singly generated by a self-adjoint operator $\widetilde{A}$ which must also be a generator for $\mathcal{A}^W(F)$. By theorem 1, page 895, \ref{Dunford}, there exists a spectral measure $$
%\end{remark}

\vskip1cm

\section{Characterization of Semi-spectral Measures which admit strong Feller Markov Kernels}
\noindent
In the last section we proved that each commutative semispectral measure admits a strong Markov kernel $\mu$ such that $\mu_{\Delta}$ is a continuous function for each $\Delta\in\mathcal{R(S)}$ where, $\mathcal{R(S)}$ is a ring which generates the Borel $\sigma$-algebra $\mathcal{B}(\mathbb{R})$.

\noindent
In the present section we characterize the commutative semispectral measures for which the Markov kernel $\mu$, whose existence was proved in theorem \ref{Cha}, is such that $\mu_{\Delta}$ is continuous for each $\Delta\in\mathcal{B}(\mathbb{R})$.
Whenever such a Markov kernel exists, we say that the semispectral measure admits a strong Feller Markov kernel. In particular, we prove that a commutative semispectral measure $F$ admits a strong Feller Markov kernel  if and only if $F$ is uniformly continuous. %First, we restrict ourselves to semispectral measures with spectrum in $[0,1]$. In the appendix A, we extend the results to the case of semispectral measures with unbounded spectrum (see theorem \ref{uniA}).  

\begin{definition}
Let $F:\mathcal{B}(\mathbb{R})\to\mathcal{F(H)}$.  Let $\Delta=\cup_{i=1}^{\infty}\Delta_i$,  $\Delta_i\cap\Delta_j=\emptyset$. If 
$$\lim_{n\to\infty}\sum_{i=1}^{n}F(\Delta_i)=F(\Delta)$$
 in the uniform operator topology then we say that $F$ is uniformly continuous.
\end{definition}
\noindent
Notice that the term uniformly continuous derives from the fact that the $\sigma$-additivity of $F$  in the uniform operator topology is equivalent to the continuity in the uniform operator topology. Analogously, the $\sigma$-additivity of $F$ in the weak operator topology is equivalent to the continuity of $F$ in the weak operator topology \cite{Berberian}. 
\begin{definition}
A Markov kernel $\mu_{(\cdot)}(\cdot):[0,1]\times\mathcal{B}(\mathbb{R})\to[0,1]$ is said to be strong Feller if  $\mu_{\Delta}$ is a continuous function for each $\Delta\in\mathcal{B}(\mathbb{R})$.
\end{definition}
\begin{definition}
We say that a commutative POVM admits a strong Feller Markov kernel if there exists a strong Feller Markov kernel $\mu$ such that $F(\Delta)=\int \mu_{\Delta}(\lambda)\,dE_{\lambda}$, where $E$ is the sharp reconstruction of $F$.
\end{definition}

\noindent
In order to prove the main theorem of the section we need the following lemma.
\begin{lemma}\label{lem4}
Let $F$ be uniformly continuous. Let $\mu$ be a weak Markov kernel and $(F,A,\mu)$ a von Neumann triplet. Suppose that $\mu_{\Delta}$ is continuous for each $\Delta\in\mathcal{R(S)}$. Then, for each $\lambda\in\sigma(A)$, $\mu_{(\cdot)}(\lambda)$ is $\sigma$-additive on $\mathcal{R(S)}$.
\end{lemma}
\begin{proof}
Let $\Delta,\Delta_i\in\mathcal{R(S)}$, $\Delta_i\cap\Delta_j=\emptyset$, $\cup_{i=1}^{\infty}\Delta_i=\Delta$. Then, 
\begin{equation*}
\mathbf{0}=u-\lim_{n\to\infty}\big(F(\Delta)-F(\cup_{i=1}^n\Delta_i)\big)=u-\lim_{n\to\infty}\int\big(\mu_{\Delta}(\lambda)-\sum_{i=1}^n\mu_{\Delta_i}(\lambda)\big)\,dE_{\lambda}.
\end{equation*}
By the uniform continuity  of $F$ and theorem 1, page 895, in Ref. \cite{Dunford}, it follows that, $\forall\epsilon>0$, there exists a number $\bar{n}\in\mathbb{N}$, such that $n>\bar{n}$ implies,
\begin{align}\label{conti}
\|\mu_{\Delta}(\lambda)-\sum_{i=1}^n\mu_{\Delta_i}(\lambda)\|_{\infty}=\|\int\big(\mu_{\Delta}(\lambda)-\sum_{i=1}^n\mu_{\Delta_i}(\lambda)\big)\,dE_{\lambda}&\|\\
=\|F(\Delta)-F(\cup_{i=1}^n\Delta_i)&\|\leq\epsilon.\notag
\end{align}
By equation (\ref{conti}),
\begin{equation*}
|\mu_{\Delta}(\lambda)-\sum_{i=1}^n\mu_{\Delta_i}(\lambda)|\leq\epsilon,\,\,\,\,\forall\lambda\in\sigma(A).
\end{equation*}
%Hence, the sequence of continuous functions $\{\mu_{\Delta_i}\}_{i\in\mathbb{N}}$ converges uniformly to $\mu_{\Delta}$ and this proves the continuity of $\mu_{\Delta}$.
\end{proof}
\vskip.3cm
\noindent
%Lemma \ref{lem4} assures that $\mu$ is $\sigma$-additive on $\mathcal{R(S)}$. 

\noindent

\begin{theorem}\label{uni}
A commutative POVM $F:\mathcal{B}(\mathbb{R})\to\mathcal{F(H)}$ admits a strong Feller Markov kernel if and only if it is uniformly continuous.
\end{theorem}
\begin{proof}
\noindent
Suppose $F$ uniformly continuous. By theorem \ref{weak}, there is a weak Markov kernel $\mu:\sigma(A)\times\mathcal{B}(\mathbb{R})\to[0,1]$ such that $\mu_{\Delta}(\cdot)$ is continuous for every $\Delta\in\mathcal{R(S)}$ and a self-adjoint operator $A$ such that $(F,A,\mu)$ is a von Neumann triplet.  By lemma \ref{lem4}, $\mu$ is $\sigma$-additive on $\mathcal{R(S)}$. Therefore (see proposition 2 in Ref. \cite{B1}), the map $\mu:\sigma(A)\times\mathcal{R(S)}\to[0,1]$ can be extended  to a Markov kernel $\widetilde{\mu}:\sigma(A)\times\mathcal{B}(\mathbb{R})\to[0,1]$ whose restriction to $\mathcal{R(S)}$ coincides with $\mu$ and such that $F(\Delta)=\widetilde{\mu}_{\Delta}(A)$.

Now we prove that $\widetilde{\mu}_{\Delta}$ is continuous for each $\Delta\in\mathcal{B}(\mathbb{R})$. We proceed by steps.

\textbf{1)} $\widetilde{\mu}$ is continuous for each open interval. For each open interval, there exists an increasing family of sets $\Delta_i\in\mathcal{S}$ such that $\Delta_i\uparrow\Delta$. Indeed, if $\Delta=(a,b)$, $a,b\in\overline{R}$, the family of sets $\{(a_i,b_i)\in\mathcal{S}\}_{i\in\mathbb{N}}$ such that $a_i>a_{i+1}> a$, $\lim_{i\to\infty}a_i=a$, $b_i<b_{i+1}$, $\lim_{i\to\infty}b_i=b$, is increasing and   $\cup_{i=1}^{\infty}\Delta_i=\Delta$. Then, 
\begin{equation*}
\int\widetilde{\mu}_{\Delta}(\lambda)\,dE_{\lambda}=F(\Delta)=u-\lim_{i\to\infty}F(\Delta_i)=u-\lim_{i\to\infty}\int\widetilde{\mu}_{\Delta_i}(\lambda)\,dE_{\lambda}.
\end{equation*}
By the uniform continuity  of $F$, it follows that, $\forall\epsilon>0$, there exists a number $\bar{n}\in\mathbb{N}$, such that $n,m>\bar{n}$ implies,
\begin{align}\label{conti1}
\|\widetilde{\mu}_{\Delta_n}(\lambda)-\widetilde{\mu}_{\Delta_m}(\lambda)\|_{\infty}=\|\int[\widetilde{\mu}_{\Delta_n}(\lambda)-\widetilde{\mu}_{\Delta_m}(\lambda)]\,dE_{\lambda}&\|\\
=\|F(\Delta_n)-F(\Delta_m)&\|\leq\epsilon.\notag
\end{align}
By equation (\ref{conti1}),
\begin{equation}\label{Cauchy}
|\widetilde{\mu}_{\Delta_n}(\lambda)-\widetilde{\mu}_{\Delta_m}(\lambda)|\leq\epsilon,\,\,\,\,\forall\lambda\in\sigma(A).
\end{equation}
Since $\widetilde{\mu}$ is a Markov kernel, 
$$\lim_{i\to\infty}\widetilde{\mu}_{\Delta_i}(\lambda)=\widetilde{\mu}_{\Delta}(\lambda),\quad\forall\lambda\in\sigma(A).$$
\noindent
Moreover, by equation (\ref{Cauchy}), the convergence is uniform
%sequence of continuous functions $\{\widetilde{\mu}_{\Delta_i}\}_{i\in\mathbb{N}}$ converges uniformly to $\widetilde{\mu}_{\Delta}$ 
and this proves the continuity of $\widetilde{\mu}_{\Delta}$.

\textbf{2)} $\widetilde{\mu}_{\Delta}$ is continuous for each open set. Each open set $\Delta$  is the disjoint union of a countable family of open intervals, i.e., $\Delta=\cup_{i=1}^{\infty}\Delta_i$, $\Delta_i=(a_i,b_i)$. Let us define the set $\widetilde{\Delta}_n:=\cup_{i=1}^{n}\Delta_i$. Therefore, $\widetilde{\Delta}_n\uparrow\Delta$. Moreover, $\mu_{\widetilde{\Delta}_n}$ is continuous for each $n\in\mathbb{N}$, and
\begin{equation*}
u-\lim_{i\to\infty}F(\widetilde{\Delta}_n)=F(\Delta).
\end{equation*}
Then, the same reasoning we used above allows us to conclude that the family of continuous functions $\mu_{\widetilde{\Delta}_n}$ converges uniformly to $\mu_{\Delta}$.

\textbf{3)} $\widetilde{\mu}_{\Delta}$ is continuous for each Borel set.
Let $\Delta\in\mathcal{B}(\mathbb{R})$. Since $F$ is regular, there is a decreasing sequence of open sets $G_i$, $\Delta\subset G_i$, such that 
\begin{equation*}\label{s}
s-\lim_{n\to\infty}F(\cap_{i=1}^n G_i)=s-\lim_{n\to\infty}F(G_i)=F(\Delta).
\end{equation*} 
\noindent
Moreover, by the uniform continuity of $F$,
$$u-\lim_{n\to\infty}F(\cap_{i=1}^n G_i))=F(\cap_{i=1}^\infty G_i).$$
\noindent
Therefore, $F(\Delta)=F(\cap_{i=1}^\infty G_i)$ and then, 
$$u-\lim_{n\to\infty}F(\cap_{i=1}^n G_i)=u-\lim_{n\to\infty}F(G_i)=F(\Delta).$$ 
\noindent
Then, the same reasoning we used in steps 1 and 2 allows us to conclude that the family of continuous functions $\widetilde{\mu}_{(\cap_{i=1}^n\Delta_i)}$ converges uniformly to $\widetilde{\mu}_{\Delta}$ and then the continuity of $\widetilde{\mu}_{\Delta}$.

In order to prove the second part of the theorem we show that the existence of a strong Feller Markov kernel implies the uniform continuity of $F$. Suppose that there exists a strong Feller Markov kernel $\mu$ such that $F(\Delta)=\mu_{\Delta}(\lambda)$. Since $\mu$ is a Markov kernel it is $\sigma$-additive. Then,
\begin{equation*}
\lim_{n\to\infty}\big(\mu_{\Delta}(\lambda)-\sum_{i=1}^{n}\mu_{\Delta_i}(\lambda)\big)=0, \quad\lambda\in\sigma(A).
\end{equation*}
\noindent
where, $\Delta,\Delta_i\in\mathcal{B}([0,1])$, $\cup_{i=1}^{\infty}\Delta_i=\Delta$.

\noindent
By hypothesis,
$$\mu_{\Delta}(\lambda)-\sum_{i=1}^{n}\mu_{\Delta_i}(\lambda)\in\mathcal{C}(\sigma(A)),\quad\forall n\in\mathbb{N}.$$
\noindent
Then, by theorem \ref{lem7} in appendix B,
$$u-\lim_{n\to\infty}\big(\mu_{\Delta}(\lambda)-\sum_{i=1}^{n}\mu_{\Delta_i}(\lambda)\big)=0.$$
\noindent
 By theorem 1, page 895, in Ref. \cite{Dunford}, $\|F(\Delta)\|=\|\mu_{\Delta}\|_{\infty}$, hence
\begin{align*}
%0&\leq\lim_{n\to\infty}\| F(\Delta)\|-\|F(\cup_{i=1}^n\Delta_i)\|\leq
\lim_{n\to\infty}\| F(\Delta)-F(\cup_{i=1}^n\Delta_i)\|=\lim_{n\to\infty}\|\mu_{\Delta}-\sum_{i=1}^n\mu_{\Delta_i}\|_{\infty}=0.
\end{align*} 
\noindent
which proves that $F$ is uniformly continuous.
\end{proof}

\begin{example}\label{example2}
Let us consider the following unsharp position observable 
\begin{align}\label{PositionExample2}
Q^f(\Delta)&:=\int_{[0,1]}\mu_{\Delta}(x)\,dQ_x,\quad\Delta\in\mathcal{B}(\mathbb{R}),\\
\mu_{\Delta}(x)&:=\int_{\mathbb{R}}\chi_{\Delta}(x-y)\, f(y)\,dy,\quad x\in[0,1]\notag
\end{align}
\noindent
where, $f$ is a bounded, continuous function such that $f(y)=0$, $y\notin [0,1]$ and  

$$\int_{[0,1]} f(y)\,dy=1,$$
\noindent
 and $Q_x$ is the spectral measure corresponding to the position operator
\begin{align*}
Q:L^2([0,1])&\to L^2([0,1])\\
\psi(x)&\mapsto (Q\psi)(x):=x\psi(x)
\end{align*}
\noindent
Notice that, for each $\Delta\in\mathcal{B}({\mathbb{R}})$, $\mu_{\Delta}:[0,1]\to[0,1]$ is continuous. Indeed, by the uniform continuity of $f$, for each $\epsilon>0$, there is a $\delta>0$ such that $\vert x-x'\vert\leq\delta$ implies $\vert f(x-y)-f(x'-y)\vert\leq\epsilon$, for each $y$. Therefore,
\begin{align*}
\vert\mu_{\Delta}(x)-\mu_{\Delta}(x')\vert&=\Big\vert\int_{\mathbb{R}}\chi_{\Delta}(x-y)\, f(y)\,dy-\int_{\mathbb{R}}\chi_{\Delta}(x'-y)\, f(y)\,dy\Big\vert\\
&=\Big\vert\int_{\Delta} [f(x-y)-f(x'-y)]\,dy\Big\vert\leq\epsilon\int_{\Delta\cap[-1,1]}\,dy\leq2\epsilon
\end{align*}
\noindent
By theorem \ref{uni} and the continuity of $\mu_{\Delta}$, $\Delta\in\mathcal{B}(\mathbb{R})$, $Q^f$ is uniformly continuous. That can be proved as follows. Suppose $\Delta_i\downarrow\Delta$ and $f(y)\leq M$, $y\in\mathbb{R}$. Since, for each $x\in[0,1]$, 
$$\mu_{\Delta_i-\Delta}(x)=\int_{\Delta_i-\Delta} f(x-y)\,dy\leq M\int_{(\Delta_i-\Delta)\cap[-1,1]}dx$$
\noindent
we have that, for each $\psi\in\mathcal{H}$, $|\psi|^2=1$,
\begin{align*}
\langle\psi,Q^f(\Delta_i-\Delta)\psi\rangle=\int_{[0,1]}\mu_{\Delta_i-\Delta}(x)\,|\psi|^2(x)\,dx\leq M\int_{(\Delta_i-\Delta)\cap[-1,1]}dx
\end{align*}
\noindent
which proves the uniform continuity of $Q^f$.
\end{example}

\noindent
In the case of  uniformly continuous POV measures, we can prove a necessary condition for the norm-1-property which has been recently used in Ref. \cite{B11} in order to study the localization in phase space of massless relativistic particles.
\begin{definition}[\cite{Heinonen}]
A semispectral measure $F$ has the norm-1-property if $\|F(\Delta)\|=1$, for each $\Delta\in\mathcal{B}(\mathbb{R})$ such that $F(\Delta)\neq\mathbf{0}$. 
\end{definition}
\begin{theorem}\label{norm1}
Let $F$ be uniformly continuous. Then, $F$ has the norm-1-property only if $\|F(\{\lambda\})\|\neq 0$ for each $\lambda\in\sigma(F)$.
\end{theorem}
\begin{proof}
We proceed by contradiction. Suppose that $F$ has the norm-1 property and that there exists $\lambda\in\sigma(F)$, such that $\|F(\{\lambda\})\|=0$. Let $(a_i,b_i)\subset \mathcal{B}([0,1])$ be a sequence of open intervals such that, $a_i<\lambda<b_i$, $(a_{i+1},b_{i+1})\subset(a_i,b_i)$, $\lim_{i\to\infty}a_i=\lambda$, $\lim_{i\to\infty}b_i=\lambda$. Then, $(a_i,b_i)\downarrow\{\lambda\}$. Moreover, by the uniform continuity of $F$ and the norm-1 property, 
$$1=\lim_{i\to\infty}\|F((a_i,b_i))\|=\lim_{i\to\infty}\|F((a_i,b_i))-F(\{\lambda\})\|=0.$$

%We proceed by contradiction. Suppose that $F$ has the norm-1-property. Then, there exists $\widetilde{\Delta}\in\mathcal{S}$ such that $F(\widetilde{\Delta})\neq\mathbf{0}$ and $\|F(\widetilde{\Delta})\|=1$. Moreover, $F(\widetilde{\Delta})=\mu_{\widetilde{\Delta}}(A)=\int\mu_{\widetilde{\Delta}}(\lambda)\,d E_{\lambda}$ and $\|F\widetilde{\Delta})\|=\vert \mu_{\widetilde{\Delta}}(\lambda)\vert_{\infty}=\max_{\lambda\in[0,1]}\vert  \mu_{\widetilde{\Delta}}(\lambda)\vert$ so that, $\vert \mu_{\widetilde{\Delta}}(\lambda)\vert_{\infty}=1$. Therefore, there is a $\lambda\in[0,1]$ such that $\mu_{\widetilde{\Delta}}(\lambda)=1$. Since $\mu_{(\cdot)}(\lambda)$ is a probability measure, it must be, $\mu_{\Delta}(\lambda)=0$ for each $\Delta\in\mathcal{B}([0,1])$, $\Delta\in[0,1]-\widetilde{\Delta}$.  Thanks to the continuity of $\mu_{\widetilde{\Delta}}(\lambda)$, for each $\epsilon>0$, there exists $\delta>0$ such that $\vert\lambda-\lambda'\vert\leq\delta$ implies, $\vert \mu_ {\widetilde{\Delta}
%}(\lambda)-\mu_{\widetilde{\Delta}}(\lambda')\vert\leq\epsilon$.
\end{proof}

\begin{example}
Let $Q^f$ be as in example \ref{example2}. Theorem \ref{norm1} implies that $Q^f$ cannot have the norm-1 property. Indeed, for each $\lambda\in\mathbb{R}$, 
$$Q^f(\{\lambda\})\psi=\lim_{i\to\infty}Q^f([\lambda,\lambda_i))\psi=\lim_{i\to\infty}\mu_{[\lambda,\lambda_i)}(x)\psi(x)=0,\quad\forall\psi\in\mathcal{H}$$
\noindent
where, $\lambda,\lambda_i\in\mathbb{R}$, $\lambda_i\to\lambda$.
\end{example}
\vskip1cm

\section{Absolutely continuous semispectral measures}
%An important class of commutative POV measures which admit a continuous Markov kernel is the set of absolutely continuous POV measures. The motivation for studying such a class of POV measures is that they arise naturally in  the application of POV measure theory to physics as for example in phase space quantum mechanics \cite{Ali3,Ali4,Bush,Davies,Guz,Holevo1,P,Schroeck} where they appear as the marginals of covariant joint position-momentum observables.

\noindent
In the present section, we prove that  absolutely continuous commutative POV measures admit a strong Feller Markov kernel. Then, we apply the result to the case of the unsharp position observable. 
\begin{definition}\cite{Schroeck,Schroeck1}
A POV measure $F:\mathcal{B}(\mathbb{R})\to\mathcal{F(H)}$ is absolutely continuous with respect to a measure $\nu:\mathcal{B}(\mathbb{R})\to[0,1]$ if there exists a positive number $c$ such that $\| F(\Delta)\|\leq c\,\nu(\Delta)$, for each $\Delta\in\mathcal{B}(\mathbb{R})$.
\end{definition}
\begin{theorem}\label{abs}
Let $F$ be absolutely continuous with respect to a finite measure $\nu$. Then, $F$ is uniformly continuous.
\end{theorem}
\begin{proof}
Suppose $\Delta_i\uparrow\Delta$. We have 
\begin{align*}
\lim_{n\to\infty}\| F(\Delta)-F(\Delta_i)\|&=\lim_{n\to\infty}\| F(\Delta-\Delta_i)\|\\
&\leq c\lim_{n\to\infty}\nu(\Delta-\Delta_i)=0.
\end{align*} 
\noindent
which proves that $F$ is uniformly continuous.
\end{proof}
\begin{corollary}
Let $F$ be absolutely continuous with respect to a finite measure $\nu$. Then, $F$ is commutative if and only if there exist a self-adjoint operator $A$ and a strong Feller Markov kernel $\mu:\mathbb{R}\times\mathcal{B}(\mathbb{R})\to[0,1]$ such that:
\begin{equation}
F(\Delta)=\mu_{\Delta}(A),\quad\Delta\in\mathcal{B}(\mathbb{R})
\end{equation} 
\end{corollary}
\begin{proof}
By theorem \ref{abs}, $F$ is uniformly continuous. Then, theorem \ref{uni} implies the thesis. 
\end{proof}

\begin{example}\label{PositionExample3}
Let us consider the unsharp position operator defined as follows.
\begin{align}
Q^f(\Delta)&:=\int_{[0,1]}\mu_{\Delta}(x)\,dQ_x,\quad\Delta\in\mathcal{B}(\mathbb{R}),\\
\mu_{\Delta}(x)&:=\int_{\mathbb{R}}\chi_{\Delta}(x-y)\, f(y)\,dy,\quad x\in[0,1]\notag
\end{align}
\noindent
where, $f$ is a positive, bounded, Borel function such that $f(x)=0$, $x\notin [0,1]$, 

$$\int_{[0,1]} f(x)dx=1,$$ 
\noindent
and $Q_x$ is the spectral measure corresponding to the position operator
\begin{align*}
Q:L^2([0,1])&\to L^2([0,1])\\
\psi(x)&\mapsto Q\psi:=x\psi(x)
\end{align*}
\noindent
$Q^f$ is absolutely continuous with respect to the measure 
$$\nu(\Delta)=M\int_{\Delta\cap[-1,1]}dx.$$ 
\noindent
Indeed, for each $\psi\in\mathcal{H}$, $|\psi|^2=1$,
\begin{align*}
\langle\psi,Q^f(\Delta)\psi\rangle=\int_{[0,1]}\mu_{\Delta}(x)\,\psi^2(x)\,dx\leq M\int_{\Delta\cap[-1,1]}dx
\end{align*}
\noindent
where, the inequality
$$\mu_{\Delta}(x)=\int_{\Delta} f(x-y)\,dy\leq M\int_{\Delta\cap[-1,1]} dx$$
\noindent
has been used.

\noindent
Therefore, by theorem \ref{abs}, $Q^f(\Delta)$ is uniformly continuous. 
\end{example}
\subsection{Unsharp Position Observable}
\noindent
In the present subsection, we study an important kind of absolutely continuous POV measures, the unsharp position observables obtained as the marginals of a covariant phase space observable. 

\noindent
In the following $\mathcal{H}=L^2(\mathbb{R})$, $Q$ and $P$ denote position and momentum observables respectively and $\ast$ denotes convolution, i.e. $(f\ast g)(x)=\int f(y)g(x-y)d y$. 

\noindent
Let us consider the joint position-momentum POV measure \cite{Ali,Bush,Davies,Guz,Holevo1,Prugovecki,Schroeck1,Stulpe} 
\begin{equation*} \label{phase}
F(\Delta\times\Delta')=\int_{\Delta\times\Delta'}U_{q,p}\,\gamma\,U^*_{q,p}\,d q\, d p
\end{equation*}
where, $U_{q,p}=e^{-iqP}e^{ipQ}$ and $\gamma=\vert f\rangle\langle f\vert$, $f\in L^2(\mathbb{R})$, $\|f\|_2=1$.
The marginal
\begin{equation}
\label{approximate}
Q^{f}(\Delta):=F(\Delta\times\mathbb{R})=\int_{-\infty}^{\infty}({\bf 1}_{\Delta}\ast\vert f\vert^2)(x)\,dQ_x,\quad\Delta\in\mathcal{B}(\mathbb{R}),
\end{equation}
is an unsharp position observable.  Notice that the map $\mu_{\Delta}(x):={\bf 1}_{\Delta}\ast \vert f(x)\vert^2$ defines a Markov kernel. 

\noindent
%One can prove \cite{B4} that the position operator $Q$ is the sharp reconstruction of the approximate position observable $Q_f$. 
Moreover, $Q^f$ is absolutely continuous with respect to the Lebesgue measure. Indeed,
\begin{align*} 
Q^f(\Delta)=F(\Delta\times\mathbb{R})&=\int_{\Delta\times\mathbb{R}}U_{q,p}\,\gamma\,U^*_{q,p}\,d q\, d p\\
&=\int_{\Delta}\,dq\int_{\mathbb{R}} U_{q,p}\,\gamma\,U^*_{q,p}\,d p\\
&=\int_{\Delta}\widehat{Q}(q)\,dq\leq\int_{\Delta}\mathbf{1}\,dq
\end{align*}

\noindent
where, 
$$\widehat{Q}(q)=\int_{\mathbb{R}}U_{q,p}\,\gamma\,U^*_{q,p}\,dp.$$
\noindent
Although $Q^f$ is absolutely continuous with respect to the Lebesgue measure on $\mathbb{R}$, it is not uniformly continuous. That does not contradict theorem \ref{abs} since the Lebesgue measure on $\mathbb{R}$ is not finite. Anyway, $Q^f$ is uniformly continuous on each Borel set $\Delta$ with finite Lebesgue measure. 
%Notice that $\|f\|_2=\int \vert f(x)\vert^2 d x=1$ so that $\vert f(x)\vert^2$ is a probability density \cite{K0,Loeve} corresponding to a probability measure $\mu_f$.

\noindent
Now, we show that $Q^f$ is not in general uniformly continuous. We give the details of the following particular case.
\begin{example}[Optimal Phase Space Representation]\label{optimal}
\noindent
If we choose 
$$f^2(x)=\frac{1}{l\,\sqrt{2\,\pi}}\,e^{(-\frac{x^2}{2\,l^2})},\quad l\in\mathbb{R}-\{0\}.$$ 
\noindent
in (\ref{approximate}), we get an optimal phase space representation of quantum mechanics \cite{Prugovecki}. In this case,

\begin{align*}
Q^{f}(\Delta)&=\int_{-\infty}^{\infty}\Big(\int_{\Delta}\vert f(x-y)\vert^2)\,d y\Big)\,dQ_x\\
&=\frac{1}{l\,\sqrt{2\,\pi}}\int_{-\infty}^{\infty}\Big(\int_{\Delta}e^{-\frac{(x-y)^2}{2\,l^2}}\,d y\Big)\,dQ_x=\int_{-\infty}^{\infty}\mu_{\Delta}(x)\,dQ_x
\end{align*}
\noindent
where, 
\begin{equation}\label{PM}
\mu_{\Delta}(x)=\frac{1}{l\,\sqrt{2\,\pi}}\int_{\Delta}e^{-\frac{(x-y)^2}{2\,l^2}}\,d y 
\end{equation}
\noindent
defines a Markov kernel. 

In order to prove that $Q^f$ is not uniformly continuous we consider the family of sets $\Delta_i=(-\infty, a_i)$, $\lim_{i\to \infty}a_i=-\infty$ such that $\Delta_i\downarrow\emptyset$, and prove that $\lim_{i\to\infty}\|Q^f(\Delta_i)\|=1$. For each  $i\in\mathbb{N}$,
\begin{align*}
\lim_{x\to-\infty}\mu_{\Delta_i}(x)&=\lim_{x\to-\infty}\frac{1}{l\,\sqrt{2\,\pi}}\int_{\Delta_i}e^{-\frac{(x-y)^2}{2\,l^2}}\,d y\\
&=\lim_{x\to-\infty}\frac{1}{l\,\sqrt{2\,\pi}}\int_{(-\infty,\, a_i-x)}e^{-\frac{y^2}{2\,l^2}}\,d y=\frac{1}{l\,\sqrt{2\,\pi}}\int_{-\infty}^{\infty}e^{-\frac{y^2}{2\,l^2}}\,d y=1.  
\end{align*}
\noindent
Now, we prove that $\|F(\Delta_i)\|=1$, $i\in\mathbb{N}$. Indeed, if 
$$\psi_n=\chi_{[-n,-n+1]}(x),$$
\begin{align}\label{n}
\lim_{n\to\infty}\langle\psi_n,Q^f(\Delta_i)\psi_n\rangle&=\lim_{n\to\infty}\int_{-\infty}^{\infty}\mu_{\Delta_i}(x)\vert\psi_n(x)\vert^2\,dx\\
&=\lim_{n\to\infty}\int_{[-n,-n+1]}\mu_{\Delta_i}(x)\,dx=1.
\end{align} 
\noindent
Since, for each $\Delta\in\mathcal{B}(\mathbb{R})$, $\|Q^f(\Delta)\|\leq 1$, equation (\ref{n}) implies that $\|Q^f(\Delta_i)\|=1$, for each $i\in\mathbb{N}$. Hence,  $\lim_{i\to\infty}\|Q^f(\Delta_i)\|=1$ and $Q^f$ cannot be uniformly continuous.

It is worth noticing that although $Q^f$ is not uniformly continuous, $\mu_{\Delta}$ is continuous for each interval $\Delta\in\mathcal{B}(\mathbb{R})$. Indeed, 
\begin{align*}
\vert\mu_{\Delta}(x)-\mu_{\Delta}(x')\vert&=\frac{1}{l\,\sqrt{2\,\pi}}\Big\vert \int_{\Delta}e^{-\frac{(x-y)^2}{2\,l^2}}\,dy-\int_{\Delta}e^{-\frac{(x'-y)^2}{2\,l^2}}\,dy \Big\vert\\
&=\frac{1}{l\,\sqrt{2\,\pi}}\Big\vert\int_{\Delta_x}e^{-\frac{(y)^2}{2\,l^2}}-\int_{\Delta_{x'}} e^{-\frac{(y)^2}{2\,l^2}}\,dy\Big\vert\leq\frac{1}{l\,\sqrt{2\,\pi}}\Big\vert\int_{\overline{\Delta}} e^{-\frac{(y)^2}{2\,l^2}}\,dy\Big\vert
\end{align*}
\noindent
where, 
$$\Delta_x=\{z\in\mathbb{R}\,\vert\,z=y-x,\,y\in\Delta\},\quad\Delta_{x'}=\{z\in\mathbb{R}\,\vert\,z=y-x',\,y\in\Delta\}$$ 
\noindent
and,
$$\overline{\Delta}=(\Delta_x-\Delta_{x'})\cup(\Delta_{x'}-\Delta_x).$$

\noindent
Therefore, $\vert x-x'\vert\leq\epsilon$ implies,
\begin{equation*}
\vert\mu_{\Delta}(x)-\mu_{\Delta}(x')\vert\leq\frac{1}{l\,\sqrt{2\,\pi}}\Big\vert\int_{\overline{\Delta}} e^{-\frac{(y)^2}{2\,l^2}}\,dy\Big\vert\leq\frac{1}{l\,\sqrt{2\,\pi}}\,\int_{\overline{\Delta}}\,dy=\frac{\sqrt{2}}{l\,\sqrt{\pi}}\,\epsilon.
\end{equation*}
\end{example}

\section*{Appendices}
\appendix
\renewcommand{\thetheorem}{A\arabic{theorem}}
\setcounter{theorem}{0}

\section{$\mathcal{A}^W(F)$ coincides with the von Neumann algebra generated by $\{F(\Delta)\}_{\Delta\in\mathcal{R(S)}}$}
\noindent
We recall that $\mathcal{S}\subset\mathcal{B}(\mathbb{R})$ is the countable family of open intervals with rational end-points and $\mathcal{R(S)}$ the ring generated by $\mathcal{R}$. Theorem c, page 24, in Ref. \cite{Loeve} ensures the countability of $\mathcal{R(S)}$.   
\begin{proof}
\noindent
Let $\overline{\mathbb{R}}$ be the extended real line, $M:=\{F(\Delta)\}_{\Delta\in\mathcal{B}(\mathbb{R})}$, and $\mathcal{A}^W(F)=\mathcal{A}^W(M)$ the von Neumann algebra generated by $F$. Let $G$ denote the family of open subsets of $\mathbb{R}$ and $O:=\{F(\Delta),\,\,\Delta\in G\}$. Since the POV measure $F$ is regular, for each Borel set $\Delta$, there exists a decreasing family of open sets $G_i$ such that $F(G_i)\to F(\Delta)$ strongly. Then, $O$ is dense in $M$ and the von Neumann algebra generated by $M$ coincides with the von Neumann algebra generated by $O$. Hence, 
\begin{equation}\label{O}
\mathcal{A}^W(F)=\mathcal{A}^W(M)=\mathcal{A}^W(O).
\end{equation}
\noindent
Now, let $G_1$ denote the family of open intervals in $\mathbb{R}$. Let us consider the set $O_1=\{F(\Delta),\,\,\Delta\in G_1\}$. Each open set $\Delta$ is the disjoint union of a countable family of open intervals $\Delta_i$, i.e. $\Delta=\cup_{i=1}^{\infty}\Delta_i$. Therefore,
\begin{align*}
F(\Delta)&=F(\cup_{i=1}^{\infty}\Delta_i)=\sum_{i=1}^{\infty}F(\Delta_i)\\
&=\lim_{n\to\infty}\sum_{i=1}^n F(\Delta_i)=\lim_{n\to\infty}F(\cup_{i=1}^n\Delta_i). 
\end{align*}
Since the von Neumann algebra generated by $O_1$ contains $F(\cup_{i=1}^n\Delta_i)$, it must contain $F(\Delta)=\lim_{n\to\infty}F(\cup_{i=1}^n\Delta_i)$. Therefore, 
\begin{equation}\label{O1}
\mathcal{A}^W(O_1)=\mathcal{A}^W(O).
\end{equation}
\noindent
 Now, we prove that the von Neumann algebra $\mathcal{A}^W(O_2)$ generated by $O_2=\{F(\Delta)\}_{\Delta\in\mathcal{R(S)}}$ coincides with $\mathcal{A}^W(O_1)$.

\noindent
For each open interval $(a,b)$, $a,b\in\overline{\mathbb{R}}$, there exists a disjoint family of sets $\{\Delta_i\}_{i\in\mathbb{N}}\subset \mathcal{R(S)}$, $\Delta_i\subset (a,b)$, $i\in\mathbb{N}$, such that $(a,b)=\cup_{i=1}^{\infty}\Delta_i$. Then, 
\begin{align*}
F(a,b)&=F(\cup_{i=1}^{\infty}\Delta_i)=\sum_{i=1}^{\infty}F(\Delta_i)\\
&=\lim_{n\to\infty}\sum_{i=1}^n F(\Delta_i)=\lim_{n\to\infty}F(\cup_{i=1}^n\Delta_i). 
\end{align*}
Since the von Neumann algebra generated by $O_2$ contains $F(\cup_{i=1}^n\Delta_i)$ for each $n\in\mathbb{N}$, it must contain $F(\Delta)=\lim_{n\to\infty}F(\cup_{i=1}^n\Delta_i)$. Therefore, $\mathcal{A}^W(O_1)=\mathcal{A}^W(O_2)$ and, by equations (\ref{O}) and (\ref{O1}),
\begin{equation}\label{W}
\mathcal{A}^W(O_2)=\mathcal{A}^W(O_1)=\mathcal{A}^W(O)=\mathcal{A}^W(F)
\end{equation}
\noindent
which proves that $\mathcal{A}^W(F)$ coincides with the von Neumann algebra generated by the set $\{F(\Delta)\}_{\Delta\in\mathcal{R(S)}}$.

\end{proof}

\section{Sequences of continuous functions}
\renewcommand{\thetheorem}{B\arabic{theorem}}
\setcounter{theorem}{0}
\noindent
The following theorem is due to Dini. We give a proof based on the use of sequences.
\begin{theorem}\label{lem7} 
Let $\{f_n(\lambda)\}_{n\in\mathbb{N}}$ be a non increasing sequence of continuous functions defined on a compact set $B\subset[0,1]$ with values in $[0,1]$ and such that $f_n(\lambda)\to 0$ point-wise. Then, $f_n(\lambda)\to 0$ uniformly.
\end{theorem}
\begin{proof}
Since $f_{n+1}(\lambda)\leq f_n(\lambda)$ for each $\lambda\in B$, we have $\|f_{n+1}\|_{\infty}\leq\|f_{n}\|_{\infty}$. If $\|f_{n}\|_{\infty}\to 0$ clearly $f_n(\lambda)\to 0$ uniformly.

\noindent 
Then, suppose $\|f_{n}\|_{\infty}\to a>0$. Since $\|f_{n+1}\|_{\infty}\leq\|f_{n}\|_{\infty}$, we have $\|f_{n}\|_{\infty}\geq a$, for each $n\in\mathbb{N}$.

\noindent
 Let $\lambda_n$ be such that $f_n(\lambda_n)=\|f_n\|_{\infty}$. Since $\{\lambda_n\}$ is a bounded sequence of real numbers, there exists a convergent subsequence $\{\lambda_{n_k}\}_{k\in\mathbb{N}}$. Let $\beta$ be its limit, i.e., $\beta:=\lim_{k\to\infty}\lambda_{n_k}$. The compactness of $B$ assures that $\beta\in B$. Moreover, $\lim_{k\to\infty}f_{n_k}(\lambda_{n_k})=a$. 

\noindent
Let us consider the sequence of numbers $f_{n_k}(\beta)$. We prove that $f_{n_k}(\beta)\geq a$ for each $k\in\mathbb{N}$. We proceed by contradiction. Suppose that there exists $\bar{k}\in\mathbb{N}$ such that $f_{n_{\bar{k}}}(\beta)<a$. Then, there exists a neighborhood $I(\beta)$ of $\beta$ such that $f_{n_{\bar{k}}}(\lambda)<a$ for each $\lambda\in I(\beta)$. Moreover, since $\lambda_{n_k}\to \beta$, there exists $l\in\mathbb{N}$ such that $k>l$ implies $\lambda_{n_k}\in I(\beta)$. Take $k>\max\{\bar{k},l\}$. Then, $\lambda_{n_k}\in I(\beta)$ and $f_{n_k}(\lambda)\leq f_{n_{\bar{k}}}(\lambda)$, for each $\lambda\in B$. Therefore,  
$$f_{n_k}(\lambda_{n_k})\leq f_{n_{\bar{k}}}(\lambda_{n_k})<a$$
which contradicts the fact that $f_{n_k}(\lambda_{n_k})=\|f_{n_k}\|_{\infty}\geq a$, for each $k\in\mathbb{N}$.

\noindent
We have proved that $f_{n_k}(\beta)\geq a$, for each $k\in\mathbb{N}$. This implies that $\lim_{k\to\infty}f_{n_k}(\beta)\geq a$ and contradicts one of the hypothesis of the lemma, i.e., $\lim_{n\to\infty}f_n(\lambda)=0$ for each $\lambda\in B$.
\end{proof}

\newpage
\bibliographystyle{model1-num-names}

%\affiliationone{% in this example, two authors share an institution

%\\
  % sauthor@university.ac.uk}}
% Important: Do not put any empty line here.
%\affiliationtwo{% in this example, one author has two addresses}
  %R. Beneduci \\
  %INFN gruppo collegato Cosenza\\
   %via P. Bucci, cubo 30-B\\
%87036 Arcavacata di Rende, Cosenza (Italy)}
  % \email{hird@university.ac.uk}}
% Important: Do not put any empty line here.
% Use \affiliationthree{} for any address positioned under \affiliationone
% Use \affiliationfour{}  for any address positioned under \affiliationtwo
%\affiliationthree{~} %inserts a space to make this field empty
%\affiliationfour{%
   %Current address:\\
   %Present long-term address\\
   %Country
   %\email{t.hird@institution.edu}}
%
\end{document}